\documentclass[twoside]{article}

\usepackage[preprint]{aistats2027}

\usepackage[round]{natbib}

\usepackage{enumitem}
\usepackage[utf8]{inputenc}
\usepackage{hyperref}
\hypersetup{hypertexnames=false, colorlinks=true, linkcolor=blue!60!black, citecolor=blue!60!black, urlcolor=blue!60!black}
\usepackage{url}
\usepackage{booktabs}
\usepackage{amsfonts}
\usepackage{nicefrac}
\usepackage{xcolor}
\usepackage{amsmath, amssymb, amsthm}
\usepackage{bm}
\usepackage{algorithm}
\usepackage{algorithmic}
\usepackage{graphicx}

\usepackage{amsmath,amsfonts,bm}

\def\eqref#1{equation~\ref{#1}}

\def\1{\bm{1}}

\newcommand{\cX}{\mathcal{X}}

\def\beq{\begin{equation}}
\def\eeq{\end{equation}}
\def\ba{\begin{array}}
\def\ea{\end{array}}
\def\beann{\begin{eqnarray*}}
\def\eeann{\end{eqnarray*}}
\def\bea{\begin{eqnarray}}
\def\eea{\end{eqnarray}}

\def\BT{\begin{theorem}}
\def\ET{\end{theorem}}
\def\BL{\begin{lemma}}
\def\EL{\end{lemma}}
\def\BC{\begin{corollary}}
\def\EC{\end{corollary}}
\def\BE{\begin{example}}
\def\EE{\end{example}}
\def\BD{\begin{definition}}
\def\ED{\end{definition}}
\def\BR{\begin{remark}}
\def\ER{\end{remark}}
\def\BAS{\begin{assumption}}
\def\EAS{\end{assumption}}
\def\BI{\begin{itemize}}
\def\EI{\end{itemize}}
\def\BP{\begin{proposition}}
\def\EP{\end{proposition}}

\def\BMP{\begin{minipage}{9.5cm}}
\def\EMP{\end{minipage}}
\def\MPT{\begin{minipage}{11.5cm}}
\def\EPT{\end{minipage}}

\def\va{{\bm{a}}}
\def\vb{{\bm{b}}}
\def\vc{{\bm{c}}}
\def\vd{{\bm{d}}}

\def\vf{{\bm{f}}}

\def\vh{{\bm{h}}}

\def\vq{{\bm{q}}}

\def\vu{{\bm{u}}}
\def\vv{{\bm{v}}}
\def\vw{{\bm{w}}}
\def\vx{{\bm{x}}}
\def\vy{{\bm{y}}}
\def\vz{{\bm{z}}}
\def\vX{{\bm{X}}}
\def\vR{{\bm{R}}}

\def\mA{{\bm{A}}}

\def\mH{{\bm{H}}}
\def\mI{{\bm{I}}}
\def\mJ{{\bm{J}}}

\def\mM{{\bm{M}}}

\def\mV{{\bm{V}}}

\def\mX{{\bm{X}}}

\DeclareMathAlphabet{\mathsfit}{\encodingdefault}{\sfdefault}{m}{sl}
\SetMathAlphabet{\mathsfit}{bold}{\encodingdefault}{\sfdefault}{bx}{n}

\newcommand{\E}{\mathbb{E}}

\newcommand{\R}{\mathbb{R}}

\DeclareMathOperator*{\argmin}{arg\,min}

\newtheorem{theorem}{Theorem}
\newtheorem{lemma}{Lemma}
\newtheorem{assumption}{Assumption}
\newtheorem{example}{Example}
\newtheorem{proposition}{Proposition}
\newtheorem{corollary}{Corollary}
\newtheorem{remark}{Remark}

\begin{document}

\runningtitle{Stochastic Compositional Optimization via Hybrid Momentum Frank--Wolfe}

\twocolumn[

\aistatstitle{Stochastic Compositional Optimization via Hybrid Momentum Frank--Wolfe}

\aistatsauthor{El Mahdi Chayti \And Mohamad Fakhouri}

\aistatsaddress{Machine Learning and Optimization Laboratory (MLO), EPFL\And Independent Researcher}
]

\begin{abstract}
Stochastic compositional optimization minimizes $F(\vf(\vx), \vx)$ over a compact convex set when the inner map $\vf$ is available only through noisy samples. We propose Hybrid Momentum Stochastic Frank--Wolfe, a projection-free method for outer functions $F$ that are jointly convex and Lipschitz but not necessarily differentiable, such as maxima, CVaR, and norms. The method tracks the Jacobian and the value of $\vf$ with momentum, using a Taylor correction for the value, and passes the tracked linearization to a generalized linear minimization oracle. Under finite-variance noise, it attains an expected generalized Frank--Wolfe gap of order $K^{-1/4}$ for non-convex objectives and, when the linearized models minorize the objective, a last-iterate primal error of order $K^{-1/3}$. The analysis extends to heavy-tailed noise without clipping, with separate moment orders for values and Jacobians. With an exact oracle, the method reduces to the deterministic method of \citet{vladarean2023first}.
\end{abstract}

\section{Introduction} \label{sec:intro}

Risk-averse learning, distributionally robust optimization (DRO), multi-task learning, and fairness-constrained training give rise to objectives of the form
\begin{equation}
\begin{gathered}
    \min_{\vx \in \cX} \varphi(\vx) := F(\vf(\vx), \vx), \\
    \text{where} \quad \vf(\vx) = \E_{\xi}[\tilde{\vf}(\vx; \xi)],
\end{gathered} \label{eq:problem}
\end{equation}
where $\cX \subset \R^d$ is convex and compact, $\vf : \cX \to \R^n$ is a vector-valued inner mapping accessible only via stochastic queries, and $F : \R^n \times \cX \to \R$ is a deterministic outer function. This is the \emph{fully composite} formulation introduced by \citet{doikov2022high} and studied with a projection-free method by \citet{vladarean2023first} in the deterministic case; it generalizes both classical compositional optimization $F(\vf(\vx))$ and additive composite optimization.

Solving \eqref{eq:problem} stochastically poses a basic difficulty: since $F$ is nonlinear in its first argument, the plug-in estimator $F(\tilde\vf(\vx; \xi), \vx)$ is in general a biased estimate of $\varphi(\vx)$. A line of work culminating in \citet{chen2020solving, balasubramanian2022stochastic, xiao2022projection} addresses this by tracking the inner function with momentum-based estimators and feeding the result through the chain rule $\nabla \varphi(\vx) = \nabla \vf(\vx)^\top \nabla_{\vu} F(\vf(\vx), \vx) + \nabla_{\vx} F(\vf(\vx), \vx)$. The chain rule requires $F$ to be differentiable, which excludes the examples that motivate fully composite optimization: max-type losses, the conditional value at risk (CVaR), $\ell_1$ regularizers, and constraint-penalty formulations.

\citet{vladarean2023first} resolved this in the deterministic setting through a \emph{generalized linear minimization oracle} (GLMO), which minimizes $F$ evaluated on a linearization of $\vf$, treating $F$ as a black box rather than differentiating through it. They explicitly leave the stochastic case as an open problem.

\paragraph{Why a projection-free oracle.}
The GLMO also determines the cost of each iteration. The prox-linear subproblem of \citet{davis2019stochastic, zhang2022proxlinear} linearizes $\vf$ in the same way but adds a term $\tfrac{M}{2}\|\vx - \vy_k\|^2$. When $F$ does not depend on $\vx$, dualizing $F$ gives the same problem over $\operatorname{dom}(F^*) \subset \R^n$ in both cases, and at a fixed dual variable the inner problem over $\cX$ is a call to a \emph{linear} minimization oracle (LMO) for the GLMO and a \emph{projection} for prox-linear; on the nuclear-norm ball, one top singular pair against a projection that may need a full SVD, or at least enough singular values to locate the threshold. The GLMO still requires optimizing over the dual variable, so one call generally costs more than one linear minimization. The method suits sets whose linear minimization is much cheaper than their projection; our guarantees do not depend on this comparison.

\paragraph{Contributions.} This paper treats the stochastic case.
\begin{enumerate}
\item We propose the \emph{Hybrid Momentum Stochastic Frank--Wolfe} (HMSFW) algorithm (Algorithm~\ref{alg:hybrid_sfw}), which is, to our knowledge, the first stochastic projection-free method for fully composite problems in which a jointly convex, Lipschitz, and possibly non-differentiable outer function $F$ is applied to a nonlinear inner map known only through samples. It tracks the Jacobian and the value of $\vf$ and feeds the resulting affine model into a GLMO, which handles non-smooth $F$ directly. The value tracker has two variants: a Polyak update (\emph{Variant I}), which does not store the previous iterate, and a Taylor-corrected update (\emph{Variant II}), whose bounds do not depend on the Lipschitz constant of $\vf$.
\item We bound the expected generalized Frank--Wolfe gap at a uniformly selected iterate under noise with bounded moments of orders $r_f, r_g \in (1, 2]$ for values and Jacobians (Theorem~\ref{thm:nonconvex}); for $r_f = r_g = 2$ the rate after $K$ iterations is $K^{-1/4}$. With the Taylor tracker the exponent depends only on $r_g$ whenever $r_f \ge 2 - 1/r_g$. All our bounds have explicit constants.
\item When the linearized models minorize the objective (Assumption~\ref{assump:convex}, which implies convexity), we obtain a last-iterate primal error of order $K^{-1/3}$ under finite variance, with general moment orders in Theorem~\ref{thm:convex}, and a gap bound of the same order for an iterate selected from the second half of the run. With an exact oracle and momentum parameters equal to one, the method is the Basic Method of \citet{vladarean2023first}, and our analysis yields its $\mathcal{O}(1/K)$ convex and $\tilde{\mathcal{O}}(1/\sqrt{K})$ non-convex rates (Corollary~\ref{cor:deterministic}).
\item On group-robust matrix completion with MovieLens-100K and MovieLens-1M, under a worst-group (max) and a CVaR objective, we compare both variants with ablations, projection-based methods, and LiNASA+ICG at equal optimizer time.
\end{enumerate}

\paragraph{Related work.}
The deterministic foundation for our approach was established by \citet{vladarean2023first}, who introduced the GLMO and analyzed the deterministic Basic and Accelerated methods for fully composite optimization. In a related vein, \citet{doikov2022high} investigated higher-order methods for this same problem class.

In the stochastic regime, early work by \citet{wang2017stochastic} introduced SCGD using a two-timescale stepsize schedule. This was later streamlined by \citet{ghadimi2020single} via the single-timescale NASA algorithm. Subsequent advancements include SCSC \citep{chen2020solving}, which matches the sample complexity of SGD, and Linearized NASA \citep{balasubramanian2022stochastic}, which established level-independent rates. These methods are either unconstrained or handle constraints through projections or proximal steps.

For projection-free stochastic compositional optimization, LiNASA+ICG \citep{xiao2022projection} bounds the squared gradient mapping by $\mathcal{O}(K^{-1/2})$ using $\mathcal{O}(\sqrt{k})$ LMO calls at iteration $k$, hence $\mathcal{O}(K^{3/2})$ in total (an $\mathcal{O}(\epsilon^{-4})$ sample and $\mathcal{O}(\epsilon^{-6})$ LMO complexity for an unsquared target $\epsilon$), and it requires every level, including $F$, to be smooth. \citet{jiang2024projection} improve these complexities with STORM-type variance reduction, under smooth levels and stronger average-smoothness assumptions, and also give Frank--Wolfe gap guarantees. \citet{akhtar2021projection} treat the two-level case, also with a smooth outer function. These methods differentiate through $F$, so they do not directly apply to the non-smooth outer functions considered here without an adaptation such as smoothing.

A separate line of work treats non-smooth outer functions in the stochastic setting through a projection or proximal step: \citet{ruszczynski2021stochastic} analyses a subgradient method for nonsmooth multilevel compositions, and \citet{zhang2022proxlinear} use variance-reduced prox-linear schemes; \citet{davis2019stochastic} develop a related model-based framework for $\E_\xi h(c(\vx,\xi),\xi)$, where the expectation is outside the composition. For the additive template $\E_\xi f(\vx,\xi) + g(\mA\vx)$ with non-smooth convex $g$, stochastic projection-free methods exist \citep{locatello2019stochastic, dresdner2022faster}; they smooth $g$, whereas here the non-smooth $F$ acts on an unknown nonlinear expectation. For the separable finite sum $\tfrac1n \sum_i f_i(g_i(\vx))$, in which each block carries its own outer function, \citet{hu2023nonsmooth} propose SONX with the estimator of \citet{jiang2022multiblock}; that structure permits sampling over blocks, whereas the single outer $F$ here couples all $n$ inner values. Standard zeroth-order estimators \citep{nesterov2017random} avoid differentiating $F$ but give rates that scale with the ambient dimension, which for matrix variables is quadratic in the side length.

Stochastic gradients in deep learning can be heavy-tailed, with possibly infinite variance \citep{simsekli2019tail, zhang2020adaptive}, which has motivated analyses under bounded $r$-th moments, $r \in (1, 2]$. For smooth non-convex problems, \citet{zhang2020adaptive} show that clipped SGD attains the rate $K^{-(r-1)/(3r-2)}$, whose exponent matches their lower bound; \citet{cutkosky2021high} give high-probability bounds with clipping, and normalized momentum methods attain the rate without clipping \citep{liu2025nonconvex, hubler2025gradient}. Heavy-tailed analyses of stochastic Frank--Wolfe use clipping or robust mean estimators and hold with high probability \citep{tang2022high, sfyraki2026lions}, and \citet{han2026normalized} treat nested composite problems with normalized proximal steps. Our analysis uses neither clipping nor normalization; the bounded Frank--Wolfe step plays a similar role. It also allows value and Jacobian noise of different orders, as with Lipschitz losses on heavy-tailed data, whose per-sample gradients are bounded while the sampled losses are heavy-tailed \citep{holland2021learning}.

Finally, our variance reduction techniques draw inspiration from the standard (non-compositional) stochastic Frank--Wolfe literature. \citet{mokhtari2020stochastic} demonstrated an $\mathcal{O}(K^{-1/3})$ rate for convex objectives using momentum tracking, while \citet{yurtsever2019conditional} obtained an improved non-convex sample complexity with the SPIDER estimator; the STORM estimator of \citet{cutkosky2019momentum} gives a single-sample alternative, which we use for acceleration. To correct the bias inherent in momentum tracking, \citet{hassani2020stochastic} and \citet{zhang2019samplestochasticfrankwolfe} proposed using Hessian-vector products coupled with uniform random variables. This bias correction was recently refined by \citet{chayti2026ransomsecondordermomentumrandomized} using Beta-distributed random variables. We adapt both techniques to the fully composite setting, detailing a STORM-based acceleration in Appendix~\ref{sec:storm_extension} and a Hessian-based acceleration in Appendix~\ref{sec:hessian_extension}.

\section{Problem Setup and Assumptions} \label{sec:setup}

\paragraph{Notation.} We use $\|\cdot\|$ for the Euclidean norm (Frobenius norm for matrices), $D_{\cX} := \max_{\vx, \vy \in \cX} \|\vx - \vy\|$ for the diameter of $\cX$, $\mathbf{1} \in \R^n$ for the all-ones vector, and $\|\cdot\|_{\mathrm{op}}$ for the operator norm. 

\textbf{Generalized Frank--Wolfe gap.} For $\vy \in \cX$, the \emph{generalized Frank--Wolfe (FW) gap} is
\begin{equation} \label{eq:gen_fw_gap}
\hat\Delta(\vy) := \varphi(\vy) - \min_{\vx \in \cX} F\big(\vf(\vy) + \nabla \vf(\vy)(\vx - \vy),\, \vx\big) \ge 0,
\end{equation}
which generalizes the classical FW gap: when $F(\vu, \vx) \equiv u_1$, $\hat\Delta(\vy)$ reduces to $\max_{\vx \in \cX} \langle \nabla f_1(\vy), \vy - \vx \rangle$. Its interpretation as a stationarity measure is given after the assumptions.

\subsection{Assumptions}

\begin{assumption}[Domain compactness] \label{assump:domain}
$\cX \subset \R^d$ is convex and compact with diameter $D_{\cX} < \infty$.
\end{assumption}

\begin{assumption}[Outer function regularity] \label{assump:outer}
$F : \R^n \times \cX \to \R$ is continuous, jointly convex in $(\vu, \vx)$, and $L_F$-Lipschitz in $\vu$ uniformly in $\vx$: $|F(\vu, \vx) - F(\vu', \vx)| \le L_F \|\vu - \vu'\|$. $F$ is \emph{not} assumed to be differentiable.
\end{assumption}

\begin{assumption}[Inner function smoothness] \label{assump:inner_smooth}
$\vf : \cX \to \R^n$ is continuously differentiable with $L$-Lipschitz Jacobian: $\|\nabla \vf(\vx) - \nabla \vf(\vy)\|_{\mathrm{op}} \le L \|\vx - \vy\|$ for all $\vx, \vy \in \cX$. We write $L_J$ for the corresponding constant in Frobenius norm; one may always take $L_J = \sqrt{\min(n,d)}\, L$.
\end{assumption}

\begin{assumption}[Heavy-tailed stochastic oracle] \label{assump:noise}
There exist $r_f, r_g \in (1, 2]$ and $\sigma_f, \sigma_g \ge 0$ such that for every $\vx \in \cX$ the stochastic oracle returns unbiased estimates $\tilde\vf(\vx; \xi), \nabla\tilde\vf(\vx; \xi)$ satisfying
\[
\begin{aligned}
\E_\xi\|\tilde\vf(\vx; \xi) - \vf(\vx)\|^{r_f} &\le \sigma_f^{r_f},\\
\E_\xi\|\nabla \tilde\vf(\vx; \xi) - \nabla \vf(\vx)\|^{r_g} &\le \sigma_g^{r_g},
\end{aligned}
\]
with the Frobenius norm for the Jacobian. At iteration $k$ the sample $\xi_k$ is drawn independently of the past, so these properties also hold conditionally on the history; the value and Jacobian errors of the same sample may be correlated. Bounded variance corresponds to $r_f = r_g = 2$. We write $s_f := (r_f - 1)/r_f$, $s_g := (r_g-1)/r_g$, $a_f := s_f/(1+s_f) = (r_f-1)/(2r_f-1)$ and $a_g := (r_g-1)/(2r_g-1)$.
\end{assumption}

The following two assumptions are used in specific results.

\begin{assumption}[Lipschitz inner function] \label{assump:inner_lipschitz}
$\vf$ is $G$-Lipschitz: $\|\vf(\vx) - \vf(\vy)\| \le G \|\vx - \vy\|$ for all $\vx, \vy \in \cX$. Used only for Variant I. On compact $\cX$, Assumption~\ref{assump:inner_smooth} already implies this with $G \le \|\nabla\vf(\vx_0)\|_{\mathrm{op}} + L D_{\cX}$ for any $\vx_0 \in \cX$; we keep $G$ as a separate constant.
\end{assumption}

\begin{assumption}[Model minorization for the convex setting] \label{assump:convex}
For all $\vx, \vy \in \cX$, $F(\vf(\vy) + \nabla\vf(\vy)(\vx - \vy), \vx) \le \varphi(\vx)$. This holds when each $f_i$ is convex and $F$ is non-decreasing in $\vu$, and with equality when $\vf$ is affine. Together with Assumption~\ref{assump:outer} it implies that $\varphi$ is convex. Used only for Theorem~\ref{thm:convex}.
\end{assumption}

\paragraph{Discussion of the assumptions.}
Assumptions~\ref{assump:domain}--\ref{assump:inner_smooth} are standard in deterministic fully composite optimization \citep{vladarean2023first, doikov2022high}. Joint convexity of $F$ makes the model $m_{\vy}(\vx) := F(\vf(\vy) + \nabla\vf(\vy)(\vx - \vy), \vx)$ convex in $\vx$, while $\varphi$ may be non-convex through $\vf$. Assumption~\ref{assump:noise} reduces to bounded variance at $r_f = r_g = 2$. The condition added to the deterministic setting is the Lipschitzness of $F$ in $\vu$, which suffices for our analysis: it controls the deviation between the GLMO evaluated on the exact and on the tracked surrogates (Lemma~\ref{lem:glmo_main}). Its Euclidean constant is $L_F = 1$ for max-of-losses, $L_F = n^{\max\{1/p - 1/2,\, 0\}}$ for the $\ell_p$ norm, $L_F \le \min\{1, 1/\sqrt{\alpha n}\}$ for CVaR with tail fraction $\alpha$, and $L_F \le \sqrt{(1+\kappa)/n}$ for $\chi^2$-DRO with radius $\kappa$.

\paragraph{Stationarity.} By Assumption~\ref{assump:inner_smooth}, $|\varphi(\vy + t\vd) - m_{\vy}(\vy + t\vd)| \le \tfrac{L_F L}{2} t^2 \|\vd\|^2$, so $\varphi$ and $m_{\vy}$ have the same directional derivatives at $\vy$. Since $m_{\vy}$ is convex, $\hat\Delta(\vy) = 0$ exactly when $\varphi$ has nonnegative directional derivatives at $\vy$ toward every point of $\cX$.

\subsection{Composite Curvature}

The composite curvature constant generalizes the standard FW curvature to the non-smooth outer setting:
\begin{equation} \label{eq:curvature_def}
\begin{split}
\mathcal{S} := \sup_{\substack{\vx, \vy \in \cX \\ \gamma \in (0, 1]}} \frac{2}{\gamma^2} \Big[\, & F(\vf(\vy_\gamma), \vy_\gamma) \\
&- F(\vf(\vx) + \gamma \nabla \vf(\vx)(\vy - \vx), \vy_\gamma) \Big],
\end{split}
\end{equation}
where $\vy_\gamma := \vx + \gamma(\vy - \vx)$. Under Assumptions~\ref{assump:domain}--\ref{assump:inner_smooth}, $0 \le \mathcal{S} \le L_F L D_{\cX}^2$ (Proposition~\ref{prop:curvature} in the appendix).

\section{Algorithm} \label{sec:algorithm}

The algorithm maintains two stochastic trackers: a Polyak Jacobian tracker $\mV_k$ and a function-value tracker $\vz_k$ in one of two variants. At each step, the trackers define an affine surrogate $\tilde l_k(\vx) := \vz_k + \mV_k(\vx - \vy_k)$, fed into the GLMO.

\begin{algorithm}[h]
\caption{Hybrid Momentum Stochastic Frank--Wolfe}
\label{alg:hybrid_sfw}
\begin{algorithmic}[1]
\STATE \textbf{Input:} $\vy_0 \in \cX$, $K$, weights $\gamma_k, \beta_k, \rho_k \in (0, 1]$.
\STATE Draw $\xi_0$; $\vz_0 = \tilde\vf(\vy_0; \xi_0)$, $\mV_0 = \nabla \tilde\vf(\vy_0; \xi_0)$.
\FOR{$k = 0, \ldots, K-1$}
\IF{$k \ge 1$}
\STATE Draw $\xi_k$ at $\vy_k$, and let $\vd_k := \vy_k - \vy_{k-1}$,
\STATE $\tilde\vf_k := \tilde\vf(\vy_k; \xi_k)$, $\tilde\mJ_k := \nabla\tilde\vf(\vy_k; \xi_k)$.
\STATE \textbf{Jacobian:} $\mV_k = (1-\beta_k)\mV_{k-1} + \beta_k \tilde\mJ_k$.
\STATE \textbf{Variant I:} $\vz_k = (1-\rho_k)\vz_{k-1} + \rho_k \tilde\vf_k$.
\STATE \textbf{Variant II:} $\vz_k = (1-\rho_k)[\vz_{k-1} + \mV_k\vd_k] + \rho_k \tilde\vf_k$.
\ENDIF
\STATE \textbf{GLMO:} $\vx_{k+1} \in \argmin\limits_{\vx \in \cX} F(\vz_k + \mV_k(\vx - \vy_k), \vx)$.
\STATE $\vy_{k+1} = (1-\gamma_k)\vy_k + \gamma_k \vx_{k+1}$.
\ENDFOR
\STATE \textbf{Return:} $\vy_K$ (convex setting), or $\vy_R$ with $R$ uniform on $\{0, \ldots, K-1\}$ and independent of the run (non-convex setting).
\end{algorithmic}
\end{algorithm}

\paragraph{Variant trade-off.}
Variant I does not store the previous iterate; its drift term scales with $G \|\vy_k - \vy_{k-1}\|$. Variant II uses the first-order Taylor extrapolation $\mV_k(\vy_k - \vy_{k-1})$ to predict $\vf(\vy_k)$ from $\vf(\vy_{k-1})$, which replaces this drift by a quadratic Taylor remainder and a coupling to the Jacobian-tracking error; it stores the previous iterate, and its bounds do not involve $G$.

\paragraph{Generalized LMO.}
The GLMO step minimizes $F$ evaluated on the affine surrogate $\tilde l_k(\vx)$. It handles the non-smoothness of $F$ directly and queries no derivative of $F$, and it reduces to the classical LMO used in Frank--Wolfe methods \citep{frank1956algorithm,jaggi2013revisiting} when $F(\vu, \vx) \equiv u_1$. In general a GLMO call is a convex program whose cost depends on $F$ and $\cX$: for the max and CVaR outer functions it is a linear program when $\cX$ is polyhedral, while over a nuclear-norm ball it is a semidefinite program. Both theorems extend to a GLMO solved to an additive accuracy $\eta_k$ or a relative accuracy $\delta$ (Corollary~\ref{cor:inexact} in Appendix~\ref{sec:inexact}), which covers numerical solvers that return a certified gap.

\section{Motivating Examples}

We illustrate the GLMO on five non-smooth fully composite problems. In each case, the stochastic inner function is $\vf(\vx) = \E_\xi[\tilde\vf(\vx; \xi)]$.

\begin{example}[Robust Minimax Optimization]
Consider $n$ distinct competing objectives or losses, $f_i(\vx) = \E_{\xi_i}[\tilde{f}_i(\vx; \xi_i)]$. We aim to minimize the worst-case expected loss over a structured domain (e.g., an $\ell_1$-ball for sparsity) \citep{shalev2016minimizing}:
\begin{equation}
    \min_{\vx \in \cX} \max_{1 \le i \le n} f_i(\vx).
\end{equation}
Here, the outer function is $F(\vu) = \max_i u_i$, which is globally $1$-Lipschitz but non-differentiable. \textbf{Tractability:} The GLMO requires minimizing the maximum of $n$ affine functions over $\cX$, which is a linear program with one auxiliary scalar variable when $\cX$ is polyhedral \citep{vladarean2023first}.
\end{example}

\begin{example}[Norm-Regularized Multi-Task Learning]
In multi-task learning, the expected task losses form a vector \citep{sener2018multi}, and a norm of this vector aggregates them: the $\ell_1$ norm sums the tasks and the $\ell_\infty$ norm controls the worst task:
\begin{equation}
    \min_{\vx \in \cX} \|\vf(\vx)\|_p.
\end{equation}
Here, $F(\vu) = \|\vu\|_p$. For $p \ge 1$, $F$ is convex and $n^{\max\{1/p-1/2,0\}}$-Lipschitz, non-smooth at the origin, and not monotone in $\vu$. \textbf{Tractability:} For $p \in \{1, \infty\}$ and polyhedral $\cX$ the GLMO is a linear program; for $p = 2$ it is a second-order cone program when $\cX$ is second-order-cone representable.
\end{example}

\begin{example}[Conditional Value at Risk] \label{ex:cvar}
Risk-averse learning often minimizes the CVaR with tail fraction $\alpha \in (0,1)$, which averages the worst $\alpha$-fraction of the expected losses across $n$ tasks \citep{rockafellar2000optimization}:
\begin{gather}
    \min_{\vx \in \cX} \text{CVaR}_\alpha(\vf(\vx)), \notag\\
    \text{CVaR}_\alpha(\vu) := \inf_{\tau \in \R} \Big\{ \tau + \frac{1}{\alpha n} \sum_{i=1}^n \max(0, u_i - \tau) \Big\}.
\end{gather}
Here, $F(\vu)$ encapsulates the CVaR operator. \textbf{Tractability:} Evaluating the GLMO requires minimizing a piecewise affine convex function over $\cX$. With auxiliary variables for the terms $\max(0, \cdot)$, the GLMO is a linear program when $\cX$ is polyhedral \citep{vladarean2023first}.
\end{example}

\begin{example}[Distributionally Robust Optimization (DRO)]
In DRO, one minimizes the expected loss with respect to the worst-case distribution $\mathbf{q}$ in the probability simplex $\Delta_n$ within a $\chi^2$-divergence ball of radius $\kappa$ around the uniform distribution, $D_{\chi^2}(\vq, \tfrac{1}{n}\mathbf{1}) := n \sum_{i=1}^n (q_i - \tfrac{1}{n})^2$ \citep{namkoong2016stochastic}:
\begin{equation}
    \min_{\vx \in \cX} \max_{\mathbf{q} \in \Delta_n, D_{\chi^2}(\mathbf{q}, \frac{1}{n}\mathbf{1}) \le \kappa} \sum_{i=1}^n q_i f_i(\vx).
\end{equation}
Here, $F(\vu)$ is the supremum over the $\chi^2$ ball. \textbf{Tractability:} By convex duality the GLMO is a second-order cone program when $\cX$ is second-order-cone representable.
\end{example}

\begin{example}[Composite Regularized Optimization]
Consider a standard expected risk minimization problem augmented with a deterministic, potentially non-smooth, convex regularizer $\Xi$, finite and continuous on $\cX$ \citep{yurtsever2018conditional}:
\begin{equation}
    \min_{\vx \in \cX} \E_{\xi}[\tilde{f}(\vx; \xi)] + \Xi(\vx).
\end{equation}
Here, $f(\vx) = \E_\xi[\tilde{f}(\vx; \xi)]$, and the outer function is explicitly dependent on $\vx$: $F(u, \vx) = u + \Xi(\vx)$. \textbf{Tractability:} The GLMO computes $\argmin_{\vx \in \cX} \left[ \vz_k + \langle \mV_k, \vx - \vy_k \rangle + \Xi(\vx) \right]$. Discarding constant terms, this is $\argmin_{\vx \in \cX} \langle \mV_k, \vx \rangle + \Xi(\vx)$, the composite linear minimization oracle of generalized conditional gradient methods \citep{yurtsever2018conditional}, whose cost depends on $\Xi$.
\end{example}

Further examples covering fairness-penalized learning \citep{agarwal2018reductions, donini2018empirical} and exact penalty methods \citep{cotter2019two} are deferred to Appendix~\ref{sec:examples_appendix}.

\begin{table*}[t]
\centering
\small
\caption{Stochastic compositional optimization methods. Rates are at $r_f = r_g = 2$, for the stated stationarity measure taken unsquared, at a selected iterate (non-convex rows) or the last iterate (primal rows). Rates for different measures are not directly comparable. ``PF'' denotes projection-free. Theorem~\ref{thm:storm_rate} assumes $r_g$-average smoothness of the stochastic Jacobian.}
\label{tab:comparison}
\begin{tabular}{@{}lccccl@{}}
\toprule
Method & Stoch. & PF & Smooth $F$ & Measure & Rate \\
\midrule
Basic SCGD \citep{wang2017stochastic}  & \checkmark & --        & required & gradient norm & $\mathcal{O}(K^{-1/8})$ \\
NASA \citep{ghadimi2020single}         & \checkmark & --        & required & gradient mapping & $\mathcal{O}(K^{-1/4})$ \\
SCSC \citep{chen2020solving}           & \checkmark & --        & required & gradient norm & $\mathcal{O}(K^{-1/4})$ \\
L-NASA \citep{balasubramanian2022stochastic} & \checkmark & --  & required & gradient mapping & $\mathcal{O}(K^{-1/4})$ \\
LiNASA+ICG \citep{xiao2022projection}  & \checkmark & \checkmark & required & gradient mapping & $\mathcal{O}(K^{-1/4})$ \\
\citet{vladarean2023first} & --      & \checkmark & \emph{not required} & gen.\ FW gap & $\tilde{\mathcal{O}}(K^{-1/2})$ \\
\citet{vladarean2023first} & --      & \checkmark & \emph{not required} & primal    & $\mathcal{O}(K^{-1})$ \\
\midrule
\textbf{This work} (Thm.~\ref{thm:nonconvex})  & \checkmark & \checkmark & \emph{not required} & gen.\ FW gap & $\mathcal{O}(K^{-1/4})$ \\
\textbf{This work} (Thm.~\ref{thm:convex})     & \checkmark & \checkmark & \emph{not required} & primal    & $\mathcal{O}(K^{-1/3})$ \\
\textbf{This work} (Thm.~\ref{thm:storm_rate})& \checkmark & \checkmark & \emph{not required} & gen.\ FW gap & $\mathcal{O}(K^{-1/3})$ \\
\bottomrule
\end{tabular}
\end{table*}

\section{Convergence Analysis} \label{sec:convergence}

We present the main convergence guarantees for non-convex objectives (Theorem~\ref{thm:nonconvex}) and convex objectives (Theorem~\ref{thm:convex}), and the exact-oracle case (Corollary~\ref{cor:deterministic}). All proofs are deferred to Appendices~\ref{sec:proof_nonconvex} and \ref{sec:proof_convex}.

\subsection{Non-Convex Convergence}

\begin{theorem}[Non-convex convergence] \label{thm:nonconvex}
Suppose Assumptions~\ref{assump:domain}--\ref{assump:noise} hold, and let $\varphi^* := \min_{\cX} \varphi$, $\Phi_0 := \varphi(\vy_0) - \varphi^*$, and $R$ be uniform on $\{0, \ldots, K-1\}$. For floors $\beta, \rho \in [K^{-1}, 1]$, use the tracker weights $\beta_k = \max\{\beta, \tfrac{2}{k+2}\}$ and $\rho_k = \max\{\rho, \tfrac{2}{k+2}\}$ for $k \ge 1$, and a constant step $\gamma_k \equiv \gamma$.
\begin{enumerate}[label=(\roman*), leftmargin=*, noitemsep, topsep=2pt]
\item \emph{Variant II.} Let
\begin{equation} \label{eq:theta_def}
\theta_{\mathrm{II}} := \min\Big\{\frac{a_g}{1+a_g},\ \frac{s_f(1+a_g)}{1+s_f(2+a_g)},\ s_f\Big\}
\end{equation}
and choose $\gamma = K^{-(1-\theta_{\mathrm{II}})}$, $\beta = \gamma^{1/(1+s_g)}$, $\rho = K^{-\theta_{\mathrm{II}}/s_f}$. Then, with $\hat\Delta_k := \hat\Delta(\vy_k)$, $\E[\hat\Delta(\vy_R)] = \frac1K\sum_{k<K}\E[\hat\Delta_k] \le M_{\mathrm{II}} K^{-\theta_{\mathrm{II}}}$.
\item \emph{Variant I} (with Assumption~\ref{assump:inner_lipschitz}). Let $q := \min\{a_g, a_f\}$ and $\theta_{\mathrm{I}} := q/(1+q)$, and choose $\gamma = K^{-1/(1+q)}$, $\beta = \gamma^{1/(1+s_g)}$, $\rho = \gamma^{1/(1+s_f)}$. Then $\E[\hat\Delta(\vy_R)] \le M_{\mathrm{I}} K^{-\theta_{\mathrm{I}}}$.
\end{enumerate}
The same bounds hold for $\min_{k < K}\E[\hat\Delta_k]$, and $M_{\mathrm{II}}, M_{\mathrm{I}}$ are given in \eqref{eq:MII_explicit} and \eqref{eq:MI_explicit}. For $r_f = r_g = r$, both exponents equal $(r-1)/(3r-2)$, with $\gamma = K^{-(2r-1)/(3r-2)}$ and $\beta = \rho = K^{-r/(3r-2)}$.
\end{theorem}

\paragraph{Heavier-tailed function values.}
In terms of the orders, $\theta_{\mathrm{I}} = \min\{\tfrac{r_g-1}{3r_g-2}, \tfrac{r_f-1}{3r_f-2}\}$, the rate obtained by giving both oracles the smaller order, whereas $\theta_{\mathrm{II}} = (r_g-1)/(3r_g-2)$ whenever $r_f \ge 2 - 1/r_g$. For $r_g = 2$, a bounded $3/2$-moment of the value noise keeps $K^{-1/4}$, where the smaller order gives $K^{-1/5}$; for $r_f = 4/3$ the exponents are $4/19$ and $1/6$. The Taylor correction replaces the value drift $\gamma/\rho$ of Variant I by a coupling to the Jacobian error and a remainder of order $\gamma^2/\rho$, so a smaller $\rho$ can average heavier-tailed values.

\paragraph{Proof sketch.} The proof combines a one-step progress bound for the GLMO update (Lemma~\ref{lem:glmo_main}) with bounds on the averaged expected tracking errors (Lemmas~\ref{lem:grad_tracker_main}--\ref{lem:func_taylor_main}), obtained by unrolling the tracker recursions. Only the martingale sums of oracle noise use the moment assumptions, each at its own order, so the Jacobian error enters the Taylor tracker through its first moment alone. Telescoping and balancing the terms gives \eqref{eq:theta_def}. See Appendix~\ref{sec:proof_nonconvex}.

\paragraph{Comparison with lower bounds.}
For $r_f = r_g = r$ the exponent $(r-1)/(3r-2)$ is that of the lower bound of \citet{zhang2020adaptive} for unconstrained stochastic first-order methods in gradient norm, which is the bound of \citet{arjevani2023lower} at $r = 2$; we do not claim a matching lower bound for the constrained generalized gap with an exact GLMO.

\subsection{Convex Convergence}

\begin{theorem}[Convex convergence] \label{thm:convex}
Suppose Assumptions~\ref{assump:domain}--\ref{assump:noise} and \ref{assump:convex} hold, and let $\varphi^* := \min_{\cX}\varphi$ and $\Phi_0 := \varphi(\vy_0) - \varphi^*$. Choose
\begin{equation}\label{eq:convex_schedule}
\gamma_k = \frac{2}{k+2}, \qquad \beta_k = (k+k_g)^{-r_g/(2r_g-1)},
\end{equation}
and $\rho_k = \gamma_k$ for Variant II, with $k_g$ given in Appendix~\ref{sec:proof_convex}, and let $p := \min\{a_g, s_f\}$ for Variant II. Then
\[
\E[\varphi(\vy_K)] - \varphi^* \;\le\; A_{\mathrm{cvx}}\,(K+2)^{-p}.
\]
Moreover, for $K \ge 2$, an index $R$ drawn from $\{\lfloor K/2 \rfloor, \ldots, K-1\}$ with probabilities proportional to $\gamma_k$ satisfies $\E[\hat\Delta(\vy_R)] \le A_{\mathrm{gap}}\, K^{-p}$. For Variant I (with Assumption~\ref{assump:inner_lipschitz}) and $\rho_k = (k+k_f)^{-r_f/(2r_f-1)}$, both bounds hold with $p = \min\{a_g, a_f\}$. The constants are given in \eqref{eq:Acvx_def} and \eqref{eq:Agap_def} of Appendix~\ref{sec:proof_convex}.
\end{theorem}

For $r_f = r_g = r$ both variants give $p = (r-1)/(2r-1)$, hence $1/3$ at $r = 2$, the exponent of momentum-based stochastic Frank--Wolfe \citep{mokhtari2020stochastic}; for Variant II it stays $1/3$ with $r_f = 3/2$.

\subsection{Exact Oracle and Acceleration}

With an exact oracle the algorithm reduces to the deterministic method of \citet{vladarean2023first}.

\begin{corollary}[Exact oracle] \label{cor:deterministic}
Suppose the stochastic oracle is exact ($\sigma_f = \sigma_g = 0$ in Assumption~\ref{assump:noise}). Setting $\beta_k = \rho_k = 1$ in Algorithm~\ref{alg:hybrid_sfw} makes the trackers exact, $\mV_k = \nabla \vf(\vy_k)$ and $\vz_k = \vf(\vy_k)$ for all $k \ge 0$, and the algorithm reduces to the Basic Method of \citet{vladarean2023first}. Under Assumptions~\ref{assump:domain}--\ref{assump:inner_smooth}:
\begin{itemize}
\item \textbf{Convex case} (Assumption~\ref{assump:convex}, $\gamma_k = 2/(k+2)$): for $K \ge 1$,
$\varphi(\vy_K) - \varphi^* \le 2\mathcal{S}/(K+1)$.
\item \textbf{Non-convex case} ($\gamma_k = 1/\sqrt{k+1}$):
$\min_{0 \le k \le K-1} \hat\Delta_k \le \big(\Phi_0 + \tfrac{1}{2}\mathcal{S}(1 + \ln K)\big)/\sqrt{K}$.
\end{itemize}
These are the rates of Theorems~3.1 and 3.2 of \citet{vladarean2023first}; we prove them directly for the curvature constant \eqref{eq:curvature_def}, which differs from theirs.
\end{corollary}

With exact trackers the tracking-error terms in Lemma~\ref{lem:glmo_main} vanish; see Appendix~\ref{sec:proof_deterministic}.

\paragraph{Acceleration under stronger oracle assumptions.}
Under $r_g$-average smoothness of the stochastic Jacobian, a STORM correction \citep{cutkosky2019momentum} of the Jacobian tracker improves the exponent to $\theta_{\mathrm{acc}}$, which is \eqref{eq:theta_def} with $a_g$ replaced by $s_g$ and equals $(r-1)/(2r-1)$ for $r_f = r_g = r$ (Theorem~\ref{thm:storm_rate}); a randomized Hessian-vector correction gives the same exponent with an unbiased Hessian oracle (Appendix~\ref{sec:hessian_extension}).

\begin{figure*}[t]\centering
\includegraphics[width=\textwidth]{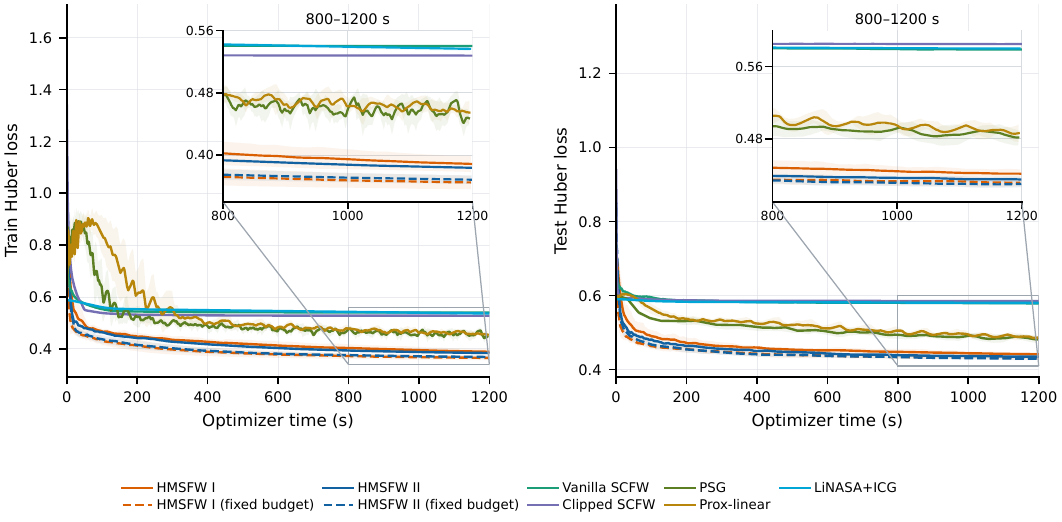}
\caption{MovieLens-1M, Max. Training left, test right; worst-group Huber loss against optimizer time. Curves are five-seed means with sample SD and end at common support for each method. Insets show the final region. Dashed HMSFW curves use a fixed budget of inner work. Appendix~\ref{sec:e1_protocol} gives the protocol and final tables.}\label{fig:e1-ml1m-max}\end{figure*}

\begin{table*}[t]
\centering
\caption{Methods with the lowest observed means among nine methods. Worst-group and average-group losses are final test Huber losses; update time is in seconds. Mean and sample SD are over five seeds. Lower is better. FB denotes fixed budget for inner pricing calls. These winners are descriptive, with no significance claim or test-based parameter selection.}
\label{tab:e1-winners}
\begingroup
\footnotesize
\setlength{\tabcolsep}{3pt}
\renewcommand{\arraystretch}{1.15}
\begin{minipage}[t]{0.49\textwidth}
\centering
\textbf{MovieLens-100K}\par\smallskip
\begin{tabular}{@{}llrr@{}}
\toprule
Metric & Winning method & Mean & SD\\
\midrule
\multicolumn{4}{@{}l}{\textbf{Max}}\\
Worst-group loss & HMSFW II & 0.5120 & 0.0046\\
Average-group loss & HMSFW I (FB) & 0.4189 & 0.0022\\
Seconds/update & PSG & 0.1921 & 0.0037\\
\addlinespace
\multicolumn{4}{@{}l}{\textbf{CVaR$_{0.3}$}}\\
Worst-group loss & HMSFW I (FB) & 0.5096 & 0.0032\\
Average-group loss & HMSFW I (FB) & 0.4160 & 0.0004\\
Seconds/update & HMSFW I (FB) & 0.1115 & 0.0240\\
\bottomrule
\end{tabular}
\end{minipage}\hfill
\begin{minipage}[t]{0.49\textwidth}
\centering
\textbf{MovieLens-1M}\par\smallskip
\begin{tabular}{@{}llrr@{}}
\toprule
Metric & Winning method & Mean & SD\\
\midrule
\multicolumn{4}{@{}l}{\textbf{Max}}\\
Worst-group loss & HMSFW II (FB) & 0.4299 & 0.0037\\
Average-group loss & HMSFW II (FB) & 0.3783 & 0.0027\\
Seconds/update & Clipped SCFW & 0.2180 & 0.0037\\
\addlinespace
\multicolumn{4}{@{}l}{\textbf{CVaR$_{0.3}$}}\\
Worst-group loss & HMSFW I (FB) & 0.4310 & 0.0034\\
Average-group loss & HMSFW II (FB) & 0.3784 & 0.0027\\
Seconds/update & Clipped SCFW & 0.3153 & 0.1487\\
\bottomrule
\end{tabular}
\end{minipage}
\endgroup
\end{table*}

\subsection{Summary of Results}

Table~\ref{tab:comparison} situates our results. Prior projection-free methods for the nested problem are deterministic \citep{vladarean2023first} or require smooth $F$ \citep{xiao2022projection, jiang2024projection}; stochastic projection-free methods with non-smooth terms cover the additive template \citep{locatello2019stochastic}, and other stochastic methods for non-smooth $F$ use projection or proximal steps (Section~\ref{sec:intro}).

Because the methods certify different quantities, the table is a guide rather than a ranking. With one sample per iteration, the accelerated exponent $K^{-1/3}$ corresponds to the $\mathcal{O}(\epsilon^{-3})$ sample complexity of SPIDER-based stochastic Frank--Wolfe \citep{yurtsever2019conditional}, which also assumes stronger smoothness of the stochastic gradients.

\section{Numerical experiments} \label{sec:experiments}

We study group-robust matrix completion on MovieLens-100K and MovieLens-1M \citep{harper2015movielens}, with two outer functions each, the maximum and CVaR$_{0.3}$, giving four problems. Nine methods share the data split, the nuclear-norm constraint, and the optimizer-time budget of each problem, with five stochastic seeds per method. Appendix~\ref{sec:e1_protocol} gives the full protocol, settings, and results.

\paragraph{Problem and metrics.}
Users are split into seven age groups. For group $g$, let $\Omega_g$ be its set of observed training entries and $\bar r_{ij}$ the rating of entry $(i,j)$ centered by the training mean. We minimize
\begin{equation}\label{eq:e1-problem}
\begin{gathered}\min_{\|\mX\|_*\le \tau} F\big(f_1(\mX),\ldots,f_7(\mX)\big),\\
f_g(\mX)=\frac{1}{|\Omega_g|}\sum_{(i,j)\in\Omega_g} h_1(X_{ij}-\bar r_{ij}),\end{gathered}
\end{equation}
where $h_1$ is the Huber loss, $h_1(e)=e^2/2$ for $|e|\le1$ and $h_1(e)=|e|-1/2$ otherwise, and $\tau$ is the nuclear-norm radius. The outer function is $F(\vu)=\max_g u_g$ or the CVaR with tail fraction $0.3$ of Example~\ref{ex:cvar}. The matrix variable is not factorized, so the problem is convex: the group losses are convex with Lipschitz gradients, and both outer functions are convex, non-decreasing, and Lipschitz, so the assumptions of Theorem~\ref{thm:convex} hold.

For both objectives, the plotted training and test losses are the maximum of the seven group losses; the final tables also report the average group loss and the training objective, which for CVaR differs from the worst-group loss. Test metrics use the last update completed within the budget.

\paragraph{Methods.}
HMSFW I and II (Variants I and II) solve the GLMO by column generation with a certified gap. Their fixed-budget versions keep the trackers but cap the number of pricing calls per update. Vanilla and Clipped SCFW (stochastic compositional Frank--Wolfe) are ablations of our method, not the SCFW method of \citet{akhtar2021projection}: they use the current minibatch values and Jacobian at every update, the latter with the Jacobian clipped. Projected subgradient (PSG) takes a projected subgradient step on the tracked model, and prox-linear solves the tracked model with a quadratic term and a fixed number of inner iterations. LiNASA+ICG \citep{xiao2022projection} runs on a quadratic smoothing of the outer function, with a GPU implementation of ICG validated against a CPU reference.

\paragraph{Schedules and scope of theory.}
The six methods with a GLMO step (HMSFW I and II, their fixed-budget versions, and Vanilla and Clipped SCFW) use $\gamma_k=2/(k+2)$, and the HMSFW variants use $\beta_k=(k+21)^{-2/3}$ and $\rho_k=2/(k+2)$, without tuning. For Variant II this is the schedule of Theorem~\ref{thm:convex} at $r_f = r_g = 2$, where $k_g = 21$, and its certified GLMO tolerance is $0.01\,\varphi(\vy_0)(k+2)^{-1/3}$, so certified HMSFW II satisfies the hypotheses of Corollary~\ref{cor:inexact}(ii). These guarantees hold at a fixed iteration count, not at the run-dependent count reached within a time budget.
Variant I uses a value-tracker weight different from that of Theorem~\ref{thm:convex}, so the theory does not cover it. The fixed-budget versions do not certify their GLMO accuracy; Corollary~\ref{cor:inexact}(ii) covers fixed-budget HMSFW II if its realized GLMO error decays as $(k+2)^{-1/3}$, which we do not measure.

\paragraph{Results.}
Table~\ref{tab:e1-winners} reports, for each problem, the methods with the lowest final worst-group test loss, the lowest final average-group test loss, and the cheapest update. An HMSFW variant has the lowest mean worst-group and average-group test loss in all four problems; in three of them the best variant uses a fixed inner budget, which the theory does not cover. In each problem, the better of the two certified variants has a lower mean worst-group test loss than every non-HMSFW method, by a margin within one standard deviation of PSG on MovieLens-100K/CVaR. All HMSFW variants reach lower worst-group test loss than Vanilla and Clipped SCFW in all four problems, which isolates the effect of the trackers. On MovieLens-1M, a certified HMSFW update costs $0.37$--$0.66$ times a PSG update and under a third of a prox-linear update; on MovieLens-100K, where the projection is cheaper, PSG has the cheapest update for the maximum (Tables~\ref{tab:e1-ml100k-results} and~\ref{tab:e1-ml1m-results}). These are comparisons of five-seed means; we make no claim of statistical significance.

Figure~\ref{fig:e1-ml1m-max} shows the training and test loss against optimizer time on MovieLens-1M with the maximum; Appendix~\ref{sec:e1-loss-curves} gives the other three problems.\footnote{The loss can increase early, since a stochastic tracked model need not give descent (Lemma~\ref{lem:glmo_main}).}

\section{Conclusion}
We introduced a stochastic projection-free algorithm for fully composite optimization with a jointly convex, Lipschitz, and possibly non-smooth outer function. Under finite-variance noise it attains an expected generalized Frank--Wolfe gap of order $K^{-1/4}$ for non-convex objectives and, under model minorization, a primal error of order $K^{-1/3}$; its Taylor-corrected tracker keeps these rates when the value noise has only a bounded $3/2$-moment. 

\textbf{Limitations and Future Work.}
First, $F$ is assumed deterministic; the \textit{doubly stochastic} regime, where $F$ is itself an expectation, remains open. Second, our guarantees hold in expectation; high-probability bounds under heavy-tailed noise are a next step. Third, our experiments are convex and have finite-variance noise, so they do not test the non-convex rate or the separate moment orders. Finally, distributed and federated versions remain to be developed.

\subsection*{AI use statement}

In this work, we used generative AI tools to edit the text, to assist in checking and revising proofs, and to draft experiment code. We have reviewed all AI-assisted work, and we take responsibility for the final content of this work, including text, claims, or artifacts produced with the aid of generative AI.

\bibliographystyle{plainnat}
\bibliography{references}

\section*{Checklist}

\begin{enumerate}

  \item For all models and algorithms presented, check if you include:
  \begin{enumerate}
    \item A clear description of the mathematical setting, assumptions, algorithm, and/or model. [Yes] Sections~\ref{sec:setup} and~\ref{sec:algorithm}.
    \item An analysis of the properties and complexity (time, space, sample size) of any algorithm. [Yes] Each iteration uses one sample and one GLMO call (Section~\ref{sec:algorithm}); the oracle complexities follow from Theorems~\ref{thm:nonconvex} and~\ref{thm:convex}.
    \item (Optional) Anonymized source code, with specification of all dependencies, including external libraries. [No] The code will be released after the review period and can be provided to reviewers on request.
  \end{enumerate}

  \item For any theoretical claim, check if you include:
  \begin{enumerate}
    \item Statements of the full set of assumptions of all theoretical results. [Yes] Assumptions~\ref{assump:domain}--\ref{assump:convex}; the extensions state their additional assumptions in Appendices~\ref{sec:storm_extension} and~\ref{sec:hessian_extension}.
    \item Complete proofs of all theoretical results. [Yes] Appendices~\ref{sec:auxiliary}--\ref{sec:hessian_extension}.
    \item Clear explanations of any assumptions. [Yes] Section~\ref{sec:setup}, discussion of the assumptions.
  \end{enumerate}

  \item For all figures and tables that present empirical results, check if you include:
  \begin{enumerate}
    \item The code, data, and instructions needed to reproduce the main experimental results (either in the supplemental material or as a URL). [No] The datasets are public and cited, and the protocol and settings are in Appendix~\ref{sec:e1_protocol}; the code will be released after the review period and can be provided to reviewers on request.
    \item All the training details (e.g., data splits, hyperparameters, how they were chosen). [Yes] Section~\ref{sec:experiments} and Appendix~\ref{sec:e1_protocol}.
    \item A clear definition of the specific measure or statistics and error bars (e.g., with respect to the random seed after running experiments multiple times). [Yes] Means and standard deviations over random seeds.
    \item A description of the computing infrastructure used. (e.g., type of GPUs, internal cluster, or cloud provider). [Yes] Appendix~\ref{sec:e1_protocol} states the cluster, GPU, precision, and timing scope. 
  \end{enumerate}

  \item If you are using existing assets (e.g., code, data, models) or curating/releasing new assets, check if you include:
  \begin{enumerate}
    \item Citations of the creator If your work uses existing assets. [Yes] MovieLens-100K and MovieLens-1M \citep{harper2015movielens}.
    \item The license information of the assets, if applicable. [Yes] Appendix~\ref{sec:e1_protocol}.
    \item New assets either in the supplemental material or as a URL, if applicable. [Not Applicable]
    \item Information about consent from data providers/curators. [Not Applicable]
    \item Discussion of sensible content if applicable, e.g., personally identifiable information or offensive content. [Not Applicable]
  \end{enumerate}

  \item If you used crowdsourcing or conducted research with human subjects, check if you include:
  \begin{enumerate}
    \item The full text of instructions given to participants and screenshots. [Not Applicable]
    \item Descriptions of potential participant risks, with links to Institutional Review Board (IRB) approvals if applicable. [Not Applicable]
    \item The estimated hourly wage paid to participants and the total amount spent on participant compensation. [Not Applicable]
  \end{enumerate}

\end{enumerate}

\clearpage
\appendix
\thispagestyle{empty}
\onecolumn
\aistatstitle{Stochastic Compositional Optimization via Hybrid Momentum Frank--Wolfe:\\
Supplementary Materials}

\section{MovieLens protocol and full results} \label{sec:e1_protocol}

\subsection{Data, objective, and statistical scope}
MovieLens-100K has 100,000 ratings from 943 users on 1,682 items. MovieLens-1M has 1,000,209 ratings from 6,040 users on 3,706 rated items. We map the rated item IDs to matrix columns. A saved random rating split, seed 20260927, assigns 80\% to training, 10\% to validation, and the remainder to test. The training mean centers all three splits; no held-out rating enters centering. The seven user-age groups are under 18, 18--24, 25--34, 35--44, 45--49, 50--55, and 56 or older. This is a random rating split, not a chronological or official MovieLens fold. The MovieLens datasets are distributed by GroupLens Research for research use under their terms of use \citep{harper2015movielens}.

Each group loss is a mean over that group's observed ratings. CVaR$_{0.3}$ uses equal group weights and averages the upper 30\% of the seven losses, with a fractional boundary group when needed. The reported average-group loss also weights groups equally; it is not the mean over all ratings. Validation and test losses use their own saved group observations. Huber loss is measured in the original rating scale after training-only centering.

The constraint is $\|\mX\|_*\le \tau$ with $\tau=0.5\sqrt{PQ}$ for a $P\times Q$ rating matrix. The common radius scale was selected on validation in a separate paired study over scales $0.5,1,2,4$, then frozen. All methods start from the same zero matrix and use minibatches of 1,024 sampled ratings per group. The final stochastic seeds are 2000--2004; tuning uses 1000 and 1001. Five seeds measure stochastic-run variation conditional on one saved split. 

\subsection{Optimizer-time budget, devices, and reporting}
Each run uses one NVIDIA Tesla V100-SXM2 GPU with 32 GB memory with CUDA and FP64 arithmetic. Optimizer-time budgets are 120 seconds for MovieLens-100K and 1,200 seconds for MovieLens-1M.

A run stops after an update crosses the optimizer budget, and only the last complete in-budget state is scored. Optimizer time includes minibatch sampling, Jacobian computation, tracker updates, GPU--CPU transfers within an update, inner solves, certificate computation, and the matrix update. It excludes data loading, initialization and setup, training, validation, and test evaluation, logging, and checkpoint storage. 

Curves are means over all five seeds and end at common support separately for each method. Bands show sample standard deviation; there is no smoothing, extrapolation, or single-seed tail. Insets retain the same data and statistics. Loss bars show mean and sample SD at the final state. Timing and completed-update bars show means and 95\% percentile-bootstrap confidence intervals from all $5^5=3125$ ordered resamples. Timing bars use logarithmic vertical axes. Confidence intervals are descriptive and have no adjustment for multiple comparisons.

\subsection{Methods}
For HMSFW I and II the Jacobian update follows Algorithm~\ref{alg:hybrid_sfw}. For the maximum and CVaR, with $\vb := \vz_k - \mV_k\vy_k$ and $\mathcal{Q}$ the simplex or the capped simplex of the CVaR dual weights, the tracked model is $\max_{\vq \in \mathcal{Q}}\{\vq^\top\vb + \langle \mV_k^\top\vq, \mX\rangle\}$ over $\|\mX\|_* \le \tau$, and every $\vq \in \mathcal{Q}$ gives the lower bound $\vq^\top\vb - \tau\|\mV_k^\top\vq\|_{\mathrm{op}}$ on its minimum. We replace $\|\mV_k^\top\vq\|_{\mathrm{op}}$ by an upper bound computed as in \citet{rump2011verified}: with $\mM := \mV_k^\top\vq$, a $P \times Q$ matrix, we take $\alpha$ slightly above a power-iteration estimate of $\|\mM\|_{\mathrm{op}}$ and accept it when a Cholesky factorization of the shifted Gram matrix $\alpha^2\mI - \mM^\top\mM$ (or $\alpha^2\mI - \mM\mM^\top$, whichever is smaller) succeeds on the GPU, since positive semidefiniteness of this matrix implies $\|\mM\|_{\mathrm{op}} \le \alpha$. The computation uses 64-bit floating-point arithmetic with error allowances for the matrix computations and rounds the bound upward, so the resulting lower bound on the GLMO minimum is valid. The certified GLMO stops when the model value at the current primal candidate exceeds it by at most $\eta_k = 0.01\,\varphi(\vy_0)(k+2)^{-1/3}$, so its additive accuracy is at most $\eta_k$. Variant I uses the Polyak value tracker; Variant II uses its Taylor correction. The empirical weights for both are
\[
\gamma_k=\frac{2}{k+2},\qquad \beta_k=(k+21)^{-2/3},\qquad \rho_k=\frac{2}{k+2}.
\]
The initialization records neutral tracker weights rather than applying a tracker update. The offset 21 is fixed.

Fixed-budget HMSFW I and II use a fixed number of inner pricing calls with a limited warm atom pool. They do not certify their per-update GLMO accuracy. Corollary~\ref{cor:inexact}(ii) applies to fixed-budget HMSFW II if its realized GLMO error decays as $(k+2)^{-1/3}$, which we do not measure. The phrase fixed budget refers to inner work per update; it does not fix the number of outer updates or change the common optimizer-time limit. The settings are listed below.

Vanilla SCFW replaces both tracked quantities by each fresh minibatch estimate, so $\beta_k=\rho_k=1$ after initialization. Clipped SCFW does the same and globally clips the stacked Jacobian in Frobenius norm at a selected threshold. Both use the same certified GLMO and outer FW schedule as the certified HMSFW runs. They isolate the effect of tracking and clipping in this code. They are ablations, not claimed reproductions of every published method called SCFW.

PSG uses the tracked values and Jacobians, a subgradient of the unsmoothed outer function, a selected constant step, and a nuclear-ball projection. Prox-linear uses the tracked affine model plus a selected quadratic coefficient and a fixed inner iteration count, with full outer update weight. The projection needs singular-value information beyond a single rank-one atom, but total update speed depends on the full implementation.

LiNASA+ICG uses quadratic smoothing of Max/CVaR, independent value/Jacobian batches, the source ICG work schedule $\lceil\sqrt{k}\rceil$, and the frozen relaxation weights $\omega_1=1$ and $\omega_k=1/\sqrt{N}$ for $k \ge 2$, where $N$ is the LiNASA horizon. We treat $N$ as a step-size parameter: runs continue beyond $N$ updates, and we score the last in-budget iterate rather than the randomized output of its horizon-$N$ guarantee. Smoothing is necessary to apply the smooth outer algorithm. It changes the optimized outer function, so we evaluate the original unsmoothed group losses for comparison. The final native GPU implementation keeps the dense gradient, atom, Gram matrix, and residual on device.

\subsection{Selection and searched parameters}
Radius selection uses validation; method settings use the mean final training objective on the tuning seeds. Test data is absent from tuning. Clipping, fixed inner work, PSG step, prox-linear coefficient/inner count, and LiNASA smoothing, beta, and horizon are searched separately. The outer HMSFW/SCFW schedule and momentum offset are not tuned. Settings are then frozen for all five final seeds. 

Fixed-budget HMSFW I initially searches calls $1,2,4,8$ with pools $5,10,30$, then extends calls to $16,32$ and explores pools $1,2$ for the MovieLens-1M Max edge. Its chosen calls/pools are $8/10$, $8/10$, $8/5$, and $2/10$ for 100K Max, 100K CVaR, 1M Max, and 1M CVaR. Fixed-budget HMSFW II chooses $4/10$, $2/10$, $4/10$, and $2/10$. Clipping thresholds are approximately $0.0032844$ for both 100K objectives, $0.00026126$ for 1M Max, and $0.00052253$ for 1M CVaR.

LiNASA extensions search beta values $0.001,0.003,0.01$ and horizons $100,300,1000$, with smoothing $0.001,0.003,0.01$ in three lower-edge problems and $0.01,0.1,1$ in 100K CVaR. A further 100K grid searches beta $0.0001,0.0003,0.001$, horizons $30,60,100$, and smoothing $0.003,0.01,0.03$ for Max or $0.001,0.003,0.01$ for CVaR.

\subsection{Results and interpretation}
\begin{table*}[p]
\centering\footnotesize
\caption{MovieLens-100K. Final test Huber losses and optimizer seconds per accepted update, mean $\pm$ sample SD over five seeds. Average-group loss weights the seven groups equally. Bold entries mark the lowest observed mean in each objective and metric, without a significance claim. FB denotes fixed budget for inner pricing calls. Results use the last complete in-budget state.}
\label{tab:e1-ml100k-results}
\setlength{\tabcolsep}{6pt}
\renewcommand{\arraystretch}{1.1}
\begin{tabular}{@{}lrrr@{}}
\toprule
Method & Worst-group test & Average-group test & Seconds/update\\
\midrule
\multicolumn{4}{@{}l}{\textbf{Max}}\\
HMSFW I & $0.5171 \pm 0.0042$ & $0.4203 \pm 0.0011$ & $0.3939 \pm 0.0873$\\
HMSFW I FB & $0.5139 \pm 0.0050$ & {\bfseries\boldmath $0.4189 \pm 0.0022$} & $0.2576 \pm 0.1456$\\
HMSFW II & {\bfseries\boldmath $0.5120 \pm 0.0046$} & $0.4201 \pm 0.0029$ & $0.6261 \pm 0.3677$\\
HMSFW II FB & $0.5168 \pm 0.0081$ & $0.4194 \pm 0.0017$ & $0.2473 \pm 0.0193$\\
Vanilla SCFW & $0.6051 \pm 0.0061$ & $0.4964 \pm 0.0021$ & $0.5627 \pm 0.3018$\\
Clipped SCFW & $0.6083 \pm 0.0051$ & $0.5083 \pm 0.0014$ & $0.3071 \pm 0.1827$\\
PSG & $0.5644 \pm 0.0012$ & $0.4599 \pm 0.0008$ & {\bfseries\boldmath $0.1921 \pm 0.0037$}\\
Prox-linear & $0.5386 \pm 0.0079$ & $0.4408 \pm 0.0075$ & $0.6445 \pm 0.2697$\\
LiNASA+ICG & $0.6063 \pm 0.0050$ & $0.5016 \pm 0.0014$ & $0.2691 \pm 0.0277$\\
\midrule
\multicolumn{4}{@{}l}{\textbf{CVaR$_{0.3}$}}\\
HMSFW I & $0.5154 \pm 0.0043$ & $0.4204 \pm 0.0007$ & $0.4171 \pm 0.0813$\\
HMSFW I FB & {\bfseries\boldmath $0.5096 \pm 0.0032$} & {\bfseries\boldmath $0.4160 \pm 0.0004$} & {\bfseries\boldmath $0.1115 \pm 0.0240$}\\
HMSFW II & $0.5169 \pm 0.0053$ & $0.4195 \pm 0.0016$ & $0.5546 \pm 0.2755$\\
HMSFW II FB & $0.5212 \pm 0.0041$ & $0.4215 \pm 0.0004$ & $0.1882 \pm 0.0182$\\
Vanilla SCFW & $0.6055 \pm 0.0060$ & $0.4970 \pm 0.0027$ & $0.5804 \pm 0.3708$\\
Clipped SCFW & $0.6075 \pm 0.0046$ & $0.5083 \pm 0.0016$ & $0.3228 \pm 0.2032$\\
PSG & $0.5164 \pm 0.0023$ & $0.4229 \pm 0.0024$ & $0.3064 \pm 0.1185$\\
Prox-linear & $0.5209 \pm 0.0042$ & $0.4243 \pm 0.0023$ & $0.4836 \pm 0.0972$\\
LiNASA+ICG & $0.5983 \pm 0.0018$ & $0.4854 \pm 0.0033$ & $0.4530 \pm 0.1307$\\
\bottomrule
\end{tabular}
\end{table*}
\begin{table*}[p]
\centering\footnotesize
\caption{MovieLens-1M. Final test Huber losses and optimizer seconds per accepted update, mean $\pm$ sample SD over five seeds. Average-group loss weights the seven groups equally. Bold entries mark the lowest observed mean in each objective and metric, without a significance claim. FB denotes fixed budget for inner pricing calls. Results use the last complete in-budget state.}
\label{tab:e1-ml1m-results}
\setlength{\tabcolsep}{6pt}
\renewcommand{\arraystretch}{1.1}
\begin{tabular}{@{}lrrr@{}}
\toprule
Method & Worst-group test & Average-group test & Seconds/update\\
\midrule
\multicolumn{4}{@{}l}{\textbf{Max}}\\
HMSFW I & $0.4413 \pm 0.0046$ & $0.3913 \pm 0.0042$ & $1.9326 \pm 0.4751$\\
HMSFW I FB & $0.4313 \pm 0.0042$ & $0.3822 \pm 0.0041$ & $0.9195 \pm 0.4837$\\
HMSFW II & $0.4349 \pm 0.0034$ & $0.3843 \pm 0.0010$ & $1.5851 \pm 0.0700$\\
HMSFW II FB & {\bfseries\boldmath $0.4299 \pm 0.0037$} & {\bfseries\boldmath $0.3783 \pm 0.0027$} & $0.7570 \pm 0.2323$\\
Vanilla SCFW & $0.5789 \pm 0.0009$ & $0.5116 \pm 0.0002$ & $0.5991 \pm 0.0120$\\
Clipped SCFW & $0.5850 \pm 0.0007$ & $0.5168 \pm 0.0001$ & {\bfseries\boldmath $0.2180 \pm 0.0037$}\\
PSG & $0.4815 \pm 0.0074$ & $0.4335 \pm 0.0080$ & $4.2942 \pm 1.2243$\\
Prox-linear & $0.4865 \pm 0.0100$ & $0.4416 \pm 0.0058$ & $6.5372 \pm 1.0817$\\
LiNASA+ICG & $0.5800 \pm 0.0006$ & $0.5156 \pm 0.0001$ & $2.7534 \pm 0.2623$\\
\midrule
\multicolumn{4}{@{}l}{\textbf{CVaR$_{0.3}$}}\\
HMSFW I & $0.4457 \pm 0.0079$ & $0.3945 \pm 0.0070$ & $2.3040 \pm 0.7633$\\
HMSFW I FB & {\bfseries\boldmath $0.4310 \pm 0.0034$} & $0.3798 \pm 0.0033$ & $0.4362 \pm 0.2394$\\
HMSFW II & $0.4365 \pm 0.0017$ & $0.3848 \pm 0.0006$ & $1.6470 \pm 0.0478$\\
HMSFW II FB & $0.4311 \pm 0.0037$ & {\bfseries\boldmath $0.3784 \pm 0.0027$} & $0.5546 \pm 0.2051$\\
Vanilla SCFW & $0.5803 \pm 0.0011$ & $0.5122 \pm 0.0004$ & $0.7250 \pm 0.1577$\\
Clipped SCFW & $0.5843 \pm 0.0004$ & $0.5164 \pm 0.0001$ & {\bfseries\boldmath $0.3153 \pm 0.1487$}\\
PSG & $0.4424 \pm 0.0012$ & $0.4042 \pm 0.0003$ & $3.4991 \pm 0.0289$\\
Prox-linear & $0.4587 \pm 0.0037$ & $0.4171 \pm 0.0033$ & $7.2589 \pm 0.9605$\\
LiNASA+ICG & $0.5775 \pm 0.0003$ & $0.5153 \pm 0.0000$ & $2.5321 \pm 0.0774$\\
\bottomrule
\end{tabular}
\end{table*}

Tables~\ref{tab:e1-ml100k-results} and~\ref{tab:e1-ml1m-results} report every method's final worst-group and average-group test loss and seconds per update. Tables~\ref{tab:e1-ml100k-supporting} and~\ref{tab:e1-ml1m-supporting} give the training objective, worst-group validation loss, and completed updates. Smaller training loss does not necessarily identify the smallest test loss: on 100K Max, fixed-budget HMSFW I minimizes the observed training mean, whereas certified HMSFW II has the lowest worst-group test mean.

\begin{table*}[p]
\centering\footnotesize
\caption{MovieLens-100K. Supporting results, mean $\pm$ sample SD over five seeds. Validation is worst-group Huber loss. The training objective is Max or CVaR as indicated. FB denotes fixed budget for inner pricing calls. Results use the last complete in-budget state.}
\label{tab:e1-ml100k-supporting}
\setlength{\tabcolsep}{6pt}
\renewcommand{\arraystretch}{1.1}
\begin{tabular}{@{}lrrr@{}}
\toprule
Method & Training objective & Validation loss & Completed updates\\
\midrule
\multicolumn{4}{@{}l}{\textbf{Max}}\\
HMSFW I & $0.3570 \pm 0.0027$ & $0.5363 \pm 0.0043$ & $314.2 \pm 55.7$\\
HMSFW I FB & $0.3530 \pm 0.0053$ & $0.5351 \pm 0.0034$ & $685.2 \pm 486.8$\\
HMSFW II & $0.3670 \pm 0.0072$ & $0.5394 \pm 0.0070$ & $247.6 \pm 120.8$\\
HMSFW II FB & $0.3599 \pm 0.0011$ & $0.5360 \pm 0.0010$ & $487.0 \pm 38.9$\\
Vanilla SCFW & $0.5109 \pm 0.0012$ & $0.5879 \pm 0.0106$ & $252.2 \pm 97.7$\\
Clipped SCFW & $0.4889 \pm 0.0020$ & $0.5981 \pm 0.0032$ & $507.8 \pm 247.8$\\
PSG & $0.3730 \pm 0.0007$ & $0.5748 \pm 0.0021$ & $624.4 \pm 12.6$\\
Prox-linear & $0.3937 \pm 0.0079$ & $0.5489 \pm 0.0070$ & $215.6 \pm 90.6$\\
LiNASA+ICG & $0.4586 \pm 0.0049$ & $0.5978 \pm 0.0045$ & $448.8 \pm 44.3$\\
\midrule
\multicolumn{4}{@{}l}{\textbf{CVaR$_{0.3}$}}\\
HMSFW I & $0.3559 \pm 0.0019$ & $0.5386 \pm 0.0051$ & $294.4 \pm 47.9$\\
HMSFW I FB & $0.3465 \pm 0.0009$ & $0.5357 \pm 0.0023$ & $1107.6 \pm 186.2$\\
HMSFW II & $0.3627 \pm 0.0049$ & $0.5399 \pm 0.0072$ & $260.2 \pm 111.7$\\
HMSFW II FB & $0.3577 \pm 0.0010$ & $0.5415 \pm 0.0049$ & $642.0 \pm 63.7$\\
Vanilla SCFW & $0.5004 \pm 0.0014$ & $0.5902 \pm 0.0063$ & $259.8 \pm 112.9$\\
Clipped SCFW & $0.4860 \pm 0.0019$ & $0.5965 \pm 0.0030$ & $499.8 \pm 256.0$\\
PSG & $0.3751 \pm 0.0122$ & $0.5336 \pm 0.0018$ & $432.8 \pm 140.7$\\
Prox-linear & $0.3766 \pm 0.0067$ & $0.5380 \pm 0.0028$ & $254.0 \pm 40.7$\\
LiNASA+ICG & $0.4236 \pm 0.0052$ & $0.5905 \pm 0.0025$ & $285.2 \pm 94.9$\\
\bottomrule
\end{tabular}
\end{table*}
\begin{table*}[p]
\centering\footnotesize
\caption{MovieLens-1M. Supporting results, mean $\pm$ sample SD over five seeds. Validation is worst-group Huber loss. The training objective is Max or CVaR as indicated. FB denotes fixed budget for inner pricing calls. Results use the last complete in-budget state.}
\label{tab:e1-ml1m-supporting}
\setlength{\tabcolsep}{6pt}
\renewcommand{\arraystretch}{1.1}
\begin{tabular}{@{}lrrr@{}}
\toprule
Method & Training objective & Validation loss & Completed updates\\
\midrule
\multicolumn{4}{@{}l}{\textbf{Max}}\\
HMSFW I & $0.3888 \pm 0.0105$ & $0.4369 \pm 0.0040$ & $651.8 \pm 161.5$\\
HMSFW I FB & $0.3651 \pm 0.0087$ & $0.4278 \pm 0.0055$ & $1690.6 \pm 957.6$\\
HMSFW II & $0.3839 \pm 0.0014$ & $0.4319 \pm 0.0010$ & $757.6 \pm 34.9$\\
HMSFW II FB & $0.3686 \pm 0.0055$ & $0.4266 \pm 0.0029$ & $1817.2 \pm 947.0$\\
Vanilla SCFW & $0.5401 \pm 0.0001$ & $0.5676 \pm 0.0016$ & $2003.2 \pm 40.5$\\
Clipped SCFW & $0.5279 \pm 0.0002$ & $0.5749 \pm 0.0004$ & $5504.6 \pm 93.5$\\
PSG & $0.4400 \pm 0.0163$ & $0.4732 \pm 0.0072$ & $296.8 \pm 79.6$\\
Prox-linear & $0.4549 \pm 0.0162$ & $0.4776 \pm 0.0096$ & $186.8 \pm 28.8$\\
LiNASA+ICG & $0.5365 \pm 0.0026$ & $0.5699 \pm 0.0006$ & $438.4 \pm 44.7$\\
\midrule
\multicolumn{4}{@{}l}{\textbf{CVaR$_{0.3}$}}\\
HMSFW I & $0.3944 \pm 0.0173$ & $0.4416 \pm 0.0075$ & $577.4 \pm 217.4$\\
HMSFW I FB & $0.3560 \pm 0.0056$ & $0.4282 \pm 0.0051$ & $3446.6 \pm 1589.5$\\
HMSFW II & $0.3824 \pm 0.0018$ & $0.4332 \pm 0.0008$ & $728.2 \pm 21.3$\\
HMSFW II FB & $0.3644 \pm 0.0051$ & $0.4285 \pm 0.0047$ & $2464.8 \pm 1027.3$\\
Vanilla SCFW & $0.5274 \pm 0.0002$ & $0.5692 \pm 0.0023$ & $1716.0 \pm 355.0$\\
Clipped SCFW & $0.5211 \pm 0.0003$ & $0.5735 \pm 0.0007$ & $4272.2 \pm 1238.8$\\
PSG & $0.4128 \pm 0.0057$ & $0.4379 \pm 0.0013$ & $342.4 \pm 2.9$\\
Prox-linear & $0.4274 \pm 0.0108$ & $0.4542 \pm 0.0039$ & $167.0 \pm 19.0$\\
LiNASA+ICG & $0.5310 \pm 0.0003$ & $0.5677 \pm 0.0004$ & $473.8 \pm 15.2$\\
\bottomrule
\end{tabular}
\end{table*}

\subsection{Additional loss curves} \label{sec:e1-loss-curves}
Figure~\ref{fig:e1-ml1m-max} in the main text gives the MovieLens-1M / Max comparison. Figures~\ref{fig:e1-ml100k-max}, \ref{fig:e1-ml100k-cvar}, and~\ref{fig:e1-ml1m-cvar} below give the other dataset/objective comparisons under the same reporting rules.

\begin{figure*}[t]\centering
\includegraphics[width=\textwidth]{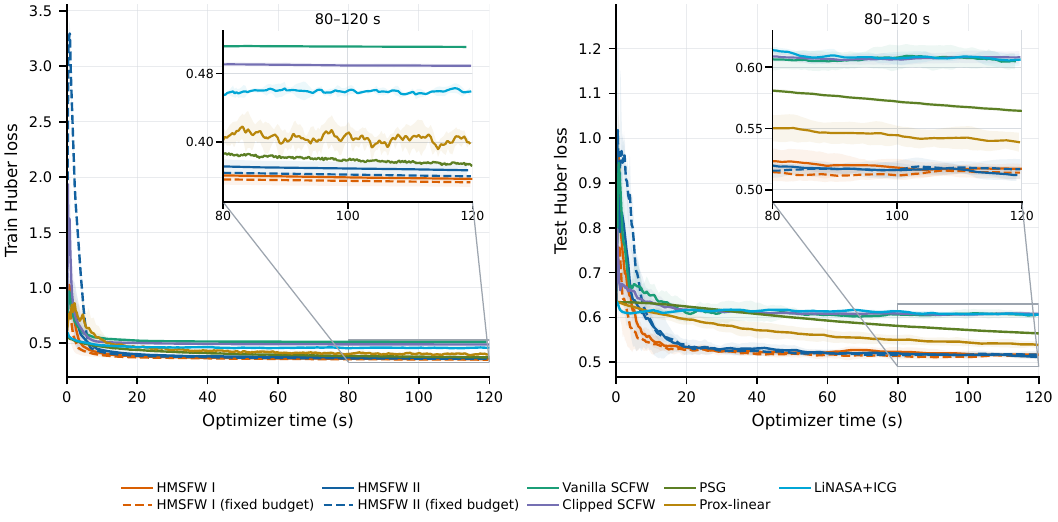}
\caption{MovieLens-100K, Max. Training left, test right; worst-group Huber loss against optimizer time. Curves are five-seed means with sample SD and end at common support for each method. Insets show the final region. Dashed HMSFW curves use fixed inner work. Appendix~\ref{sec:e1_protocol} gives the protocol and final tables.}\label{fig:e1-ml100k-max}\end{figure*}

\begin{figure*}[t]\centering
\includegraphics[width=\textwidth]{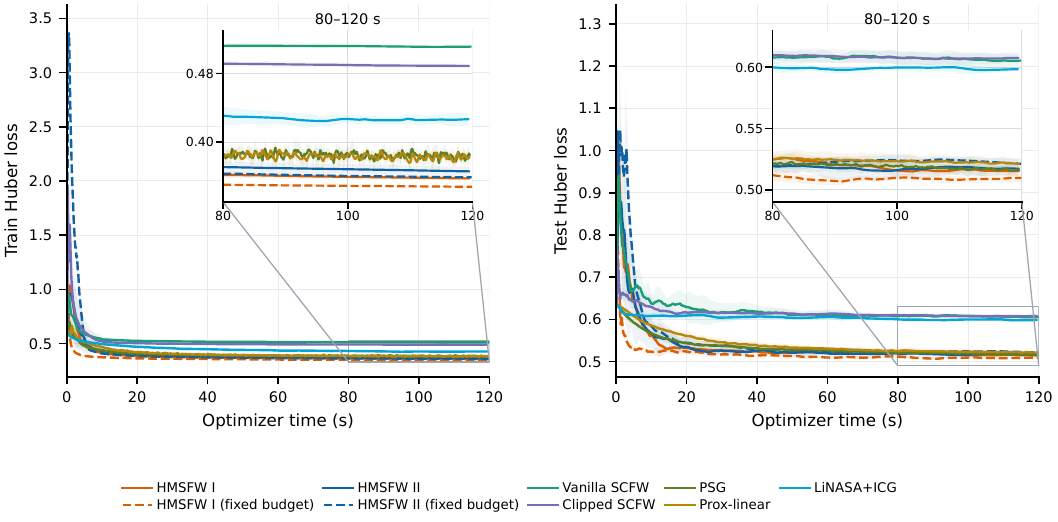}
\caption{MovieLens-100K, CVaR$_{0.3}$. Training left, test right; worst-group Huber loss against optimizer time. Curves are five-seed means with sample SD and end at common support for each method. Insets show the final region. Dashed HMSFW curves use fixed inner work. Appendix~\ref{sec:e1_protocol} gives the protocol and final tables.}\label{fig:e1-ml100k-cvar}\end{figure*}

\begin{figure*}[t]\centering
\includegraphics[width=\textwidth]{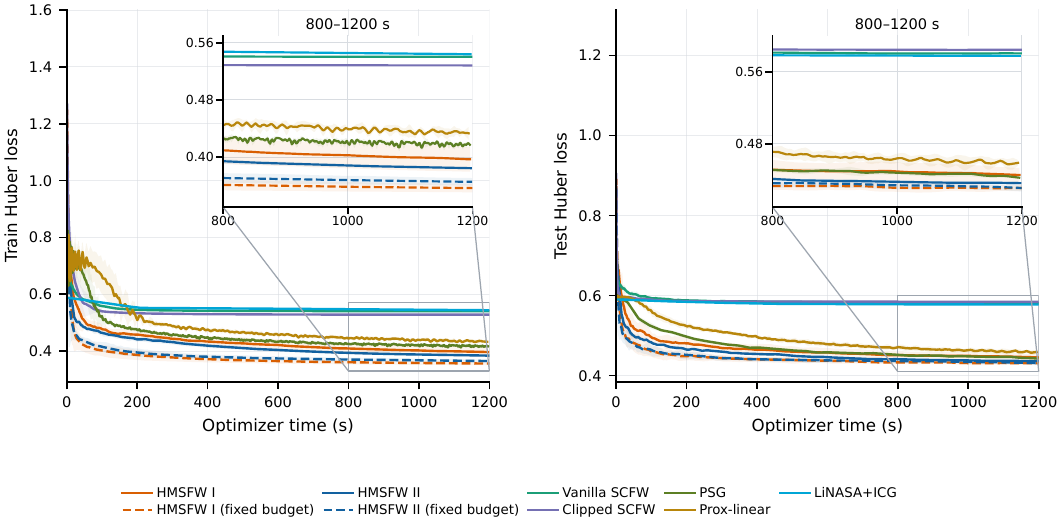}
\caption{MovieLens-1M, CVaR$_{0.3}$. Training left, test right; worst-group Huber loss against optimizer time. Curves are five-seed means with sample SD and end at common support for each method. Insets show the final region. Dashed HMSFW curves use fixed inner work.  Appendix~\ref{sec:e1_protocol} gives the protocol and final tables.}\label{fig:e1-ml1m-cvar}\end{figure*}

\subsection{Frozen settings}
\begin{table*}[p]
\centering\footnotesize
\caption{Frozen method settings for Max and CVaR$_{0.3}$. FB denotes fixed budget. Values are rounded for display; saved configuration files retain full precision.}
\label{tab:e1-settings}
\setlength{\tabcolsep}{3pt}
\renewcommand{\arraystretch}{1.15}
\begin{minipage}[t]{0.49\textwidth}
\centering
\textbf{MovieLens-100K}\par\smallskip
\begin{tabular}{@{}llrr@{}}
\toprule
Method & Parameter & Max & CVaR$_{0.3}$\\
\midrule
HMSFW I/II & $P$ & $30$ & $30$\\
\addlinespace
HMSFW I FB & $q$ & $8$ & $8$\\
 & $P$ & $10$ & $10$\\
\addlinespace
HMSFW II FB & $q$ & $4$ & $2$\\
 & $P$ & $10$ & $10$\\
\addlinespace
Clipped SCFW & $c$ & $3.284\times 10^{-3}$ & $3.284\times 10^{-3}$\\
 & $P$ & $30$ & $30$\\
\addlinespace
PSG & $\eta_{\mathrm{PSG}}$ & $373$ & $7.459\times 10^{3}$\\
\addlinespace
Prox-linear & $\lambda_{\mathrm{PL}}$ & $2.681\times 10^{-4}$ & $1.341\times 10^{-4}$\\
 & $J$ & $2$ & $2$\\
\addlinespace
LiNASA+ICG & $\mu$ & $3\times 10^{-3}$ & $1\times 10^{-3}$\\
 & $\beta_{\mathrm{L}}$ & $1\times 10^{-4}$ & $1\times 10^{-4}$\\
 & $N$ & $30$ & $30$\\
 & $\delta$ & $0.01$ & $0.01$\\
 & $P$ & $30$ & $30$\\
\bottomrule
\end{tabular}
\end{minipage}\hfill
\begin{minipage}[t]{0.49\textwidth}
\centering
\textbf{MovieLens-1M}\par\smallskip
\begin{tabular}{@{}llrr@{}}
\toprule
Method & Parameter & Max & CVaR$_{0.3}$\\
\midrule
HMSFW I/II & $P$ & $30$ & $30$\\
\addlinespace
HMSFW I FB & $q$ & $8$ & $2$\\
 & $P$ & $5$ & $10$\\
\addlinespace
HMSFW II FB & $q$ & $4$ & $2$\\
 & $P$ & $10$ & $10$\\
\addlinespace
Clipped SCFW & $c$ & $2.613\times 10^{-4}$ & $5.225\times 10^{-4}$\\
 & $P$ & $30$ & $30$\\
\addlinespace
PSG & $\eta_{\mathrm{PSG}}$ & $4.586\times 10^{4}$ & $9.006\times 10^{4}$\\
\addlinespace
Prox-linear & $\lambda_{\mathrm{PL}}$ & $2.181\times 10^{-5}$ & $1.11\times 10^{-5}$\\
 & $J$ & $2$ & $2$\\
\addlinespace
LiNASA+ICG & $\mu$ & $3\times 10^{-3}$ & $1\times 10^{-3}$\\
 & $\beta_{\mathrm{L}}$ & $1\times 10^{-3}$ & $1\times 10^{-3}$\\
 & $N$ & $100$ & $100$\\
 & $\delta$ & $0.01$ & $0.01$\\
 & $P$ & $30$ & $30$\\
\bottomrule
\end{tabular}
\end{minipage}
\par\smallskip
\begin{minipage}{0.98\textwidth}\footnotesize $P$: atom-pool limit; $q$: inner pricing calls per update; $c$: Jacobian clipping threshold; $\eta_{\mathrm{PSG}}$: PSG step size; $\lambda_{\mathrm{PL}}$: prox-linear quadratic coefficient; $J$: prox-linear inner iterations; $\mu$: LiNASA smoothing; $\beta_{\mathrm{L}}$: ICG quadratic coefficient; $N$: LiNASA step-size parameter; $\delta$: ICG accuracy parameter. Vanilla SCFW uses $P=30$ in all four problems. All FW methods use the optimized GLMO backend. LiNASA uses native GPU ICG and continues to the optimizer-time budget.\end{minipage}
\end{table*}

\subsection{Update cost}
\begin{figure}[htbp]\centering
\includegraphics[width=0.85\textwidth]{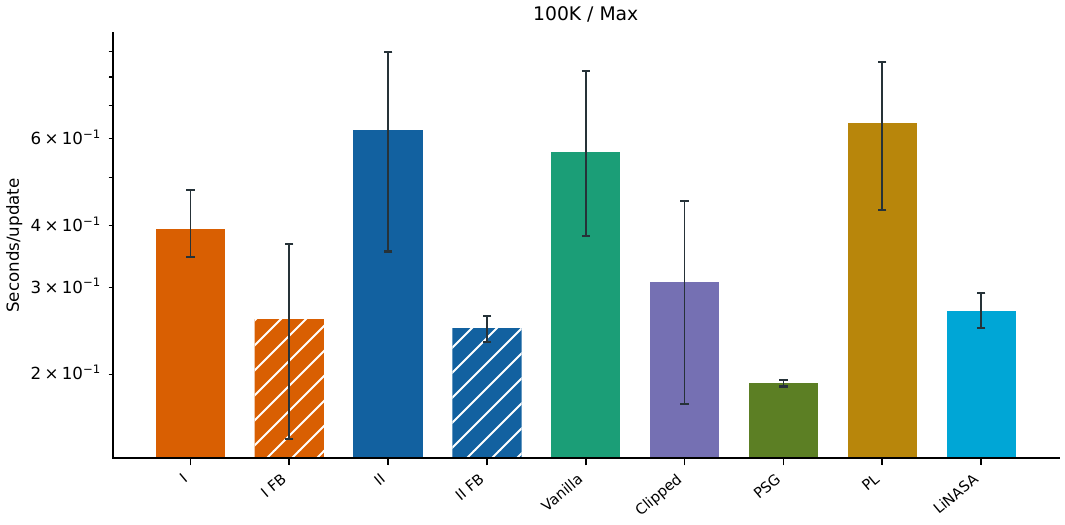}
\caption{MovieLens-100K, Max. Mean optimizer seconds per accepted update, with a logarithmic vertical axis. Error bars are 95\% percentile-bootstrap mean intervals from all 3125 ordered resamples of five seeds. I and II denote HMSFW I and II; FB denotes fixed inner work; PL denotes prox-linear.}\label{fig:e1-cost-ml100k-max}\end{figure}

\begin{figure}[htbp]\centering
\includegraphics[width=0.85\textwidth]{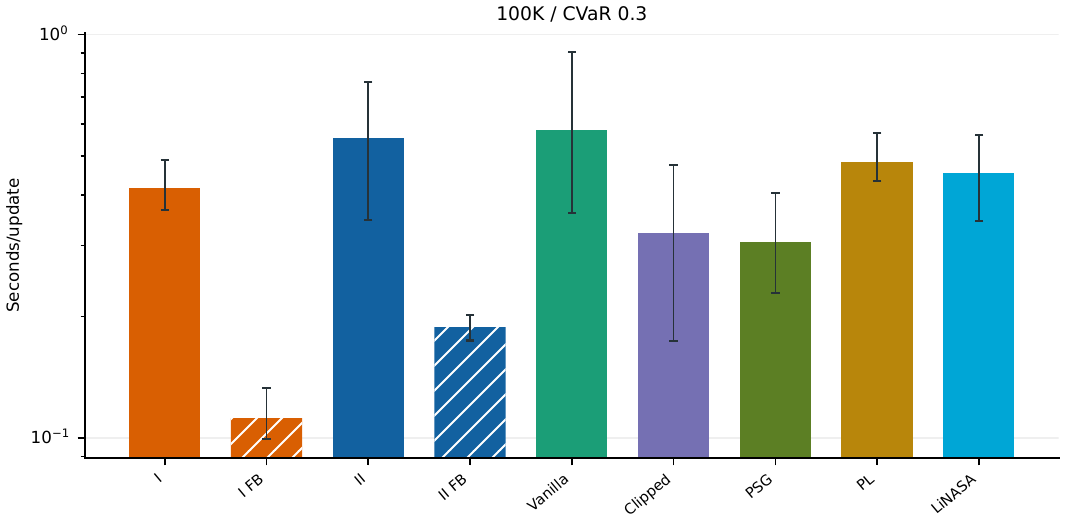}
\caption{MovieLens-100K, CVaR$_{0.3}$. Mean optimizer seconds per accepted update, with a logarithmic vertical axis. Error bars are 95\% percentile-bootstrap mean intervals from all 3125 ordered resamples of five seeds. I and II denote HMSFW I and II; FB denotes fixed inner work; PL denotes prox-linear.}\label{fig:e1-cost-ml100k-cvar}\end{figure}

\begin{figure}[htbp]\centering
\includegraphics[width=0.85\textwidth]{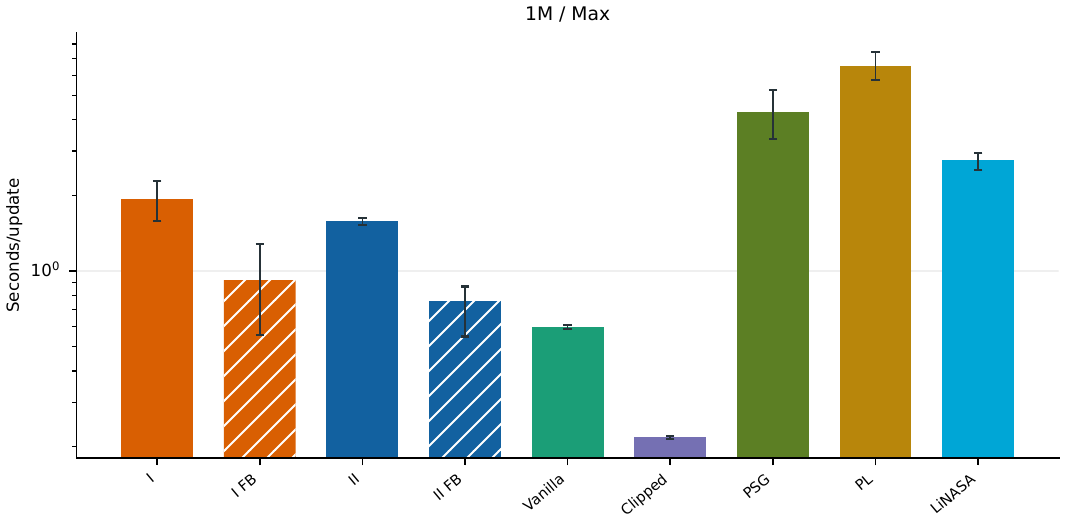}
\caption{MovieLens-1M, Max. Mean optimizer seconds per accepted update, with a logarithmic vertical axis. Error bars are 95\% percentile-bootstrap mean intervals from all 3125 ordered resamples of five seeds. I and II denote HMSFW I and II; FB denotes fixed inner work; PL denotes prox-linear.}\label{fig:e1-cost-ml1m-max}\end{figure}

\begin{figure}[htbp]\centering
\includegraphics[width=0.85\textwidth]{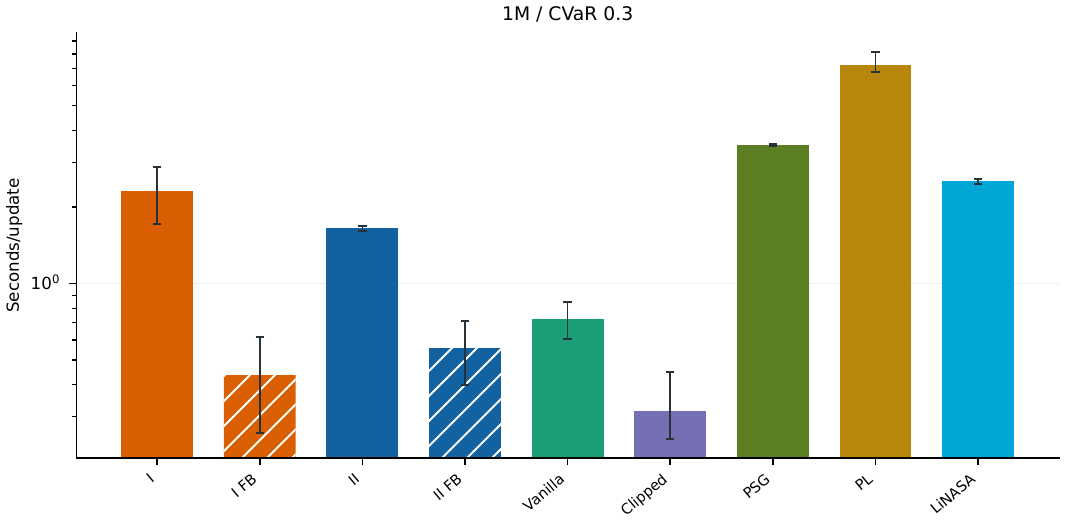}
\caption{MovieLens-1M, CVaR$_{0.3}$. Mean optimizer seconds per accepted update, with a logarithmic vertical axis. Error bars are 95\% percentile-bootstrap mean intervals from all 3125 ordered resamples of five seeds. I and II denote HMSFW I and II; FB denotes fixed inner work; PL denotes prox-linear.}\label{fig:e1-cost-ml1m-cvar}\end{figure}

\clearpage

\section{Auxiliary Inequalities} 
\label{sec:auxiliary}

Throughout the appendix we write $D := D_\cX$, and all vector and matrix norms are Euclidean (Frobenius for matrices), so every space involved is a Euclidean space.

\begin{lemma}[Norm power convexity] \label{lem:norm_convex}
For vectors $\va_1, \ldots, \va_m$ and $s \ge 1$: $\|\sum_i \va_i\|^s \le m^{s-1} \sum_i \|\va_i\|^s$.
\end{lemma}

\begin{lemma}[Subadditivity of the $r$-th root] \label{lem:subadd}
For $a_i \ge 0$ and $s \ge 1$: $(\sum_i a_i)^{1/s} \le \sum_i a_i^{1/s}$.
\end{lemma}

\begin{lemma}[Jensen for weighted sums] \label{lem:jensen}
If $w_i \ge 0$ with $\sum_i w_i \le 1$, then for any $a_i \ge 0$ and $s \ge 1$: $(\sum_i w_i a_i)^s \le \sum_i w_i a_i^s$.
\end{lemma}

\begin{lemma}[Martingale moment inequality] \label{lem:vbe}
Let $H$ be a Euclidean space, $r \in (1, 2]$, and set $c_r := 13$ for $r < 2$ and $c_2 := 1$.
\begin{enumerate}[label=(\roman*)]
\item For $\va, \vb \in H$, $\|\va + \vb\|^r \le \|\va\|^r + r\|\va\|^{r-2}\langle \va, \vb \rangle + c_r \|\vb\|^r$, where the middle term is read as $0$ when $\va = 0$.
\item Let $\{\vX_i\}_{i=1}^k$ be a martingale difference sequence in $H$ with respect to a filtration $\{\mathcal{G}_i\}$, with finite $r$-th moments, and let $w_1, \ldots, w_k$ be deterministic. Then $\E\|\sum_i w_i \vX_i\|^r \le c_r \sum_i |w_i|^r \E\|\vX_i\|^r$.
\item If $\va$ is $\mathcal{G}_{i-1}$-measurable, $\E[\vb \mid \mathcal{G}_{i-1}] = 0$, and both have finite $r$-th moments, then $\E\|\va + \vb\|^r \le \E\|\va\|^r + c_r \E\|\vb\|^r$.
\end{enumerate}
\end{lemma}

\begin{proof}
(i) At $r = 2$ this is the expansion of the square. Let $r < 2$, write $\phi(\vx) := \|\vx\|^r$, and let $\varrho := \phi(\va + \vb) - \phi(\va) - r\|\va\|^{r-2}\langle \va, \vb\rangle$. If $\|\va\| \le 2\|\vb\|$, then $\phi(\va + \vb) \le 3^r \|\vb\|^r$ and $r\|\va\|^{r-1}\|\vb\| \le r 2^{r-1}\|\vb\|^r$, so $\varrho \le (3^r + r2^{r-1})\|\vb\|^r \le 13\|\vb\|^r$. Otherwise $\|\va + t\vb\| \ge \|\va\| - \|\vb\| > \|\vb\|$ for all $t \in [0,1]$, so $\phi$ is twice differentiable on the segment, with $\|\nabla^2\phi(\vx)\|_{\mathrm{op}} \le r\|\vx\|^{r-2} \le r\|\vb\|^{r-2}$ there since $r - 2 < 0$. Taylor's formula gives $\varrho \le \tfrac{r}{2}\|\vb\|^r$.
(ii) Let $\bm{S}_i := \sum_{j \le i} w_j \vX_j$. Apply (i) with $\va = \bm{S}_{i-1}$ and $\vb = w_i \vX_i$. The middle term has zero conditional expectation, because $\bm{S}_{i-1}$ is $\mathcal{G}_{i-1}$-measurable and $\E[\vX_i \mid \mathcal{G}_{i-1}] = 0$; it is integrable by H\"older's inequality. Summing the resulting bounds over $i$ proves (ii), and (iii) is the one-step case.
\end{proof}

\begin{lemma}[Averaging kernels of the tracker weights] \label{lem:kernel}
Let $\lambda \in (0,1]$ and $\lambda_k := \max\{\lambda, 2/(k+2)\}$ for $k \ge 1$. For $0 \le i \le k$ let $P_{i,k} := \prod_{j=i+1}^k (1 - \lambda_j)$, with the empty product equal to one, and for $1 \le i \le k$ let $w_{i,k} := \lambda_i P_{i,k}$ and $q_{i,k} := (1 - \lambda_i) P_{i,k}$. Write $d_r := c_r^{1/r}$ with $c_r$ from Lemma~\ref{lem:vbe}.
\begin{enumerate}[label=(\roman*)]
\item $P_{0,k} \le \frac{2}{(k+1)(k+2)}$ and $\sum_{i=1}^k w_{i,k} = 1 - P_{0,k}$.
\item $\max_{i \le k} w_{i,k} \le \lambda_k$, $\ \sum_{i=1}^k q_{i,k} \le \lambda^{-1}$, and $\sum_{k=i}^{K-1} q_{i,k} \le \lambda^{-1}$.
\item If $\{\bm\xi_i\}$ is a martingale difference sequence with $\E[\|\bm\xi_i\|^r \mid \mathcal{G}_{i-1}] \le \sigma^r$ for some $r \in (1,2]$, then $\E\|\sum_{i=1}^k w_{i,k}\bm\xi_i\| \le d_r \sigma \lambda_k^{(r-1)/r}$.
\item If $\lambda \ge K^{-1}$ and $s \in (0,1)$, then $\frac{1}{K}\sum_{k=1}^{K-1}\lambda_k^s \le C_s\lambda^s$ with $C_s := 1 + \frac{2^s}{1-s}$, and $\frac1K\sum_{k=0}^{K-1} P_{0,k} \le \frac{2}{K}$.
\end{enumerate}
\end{lemma}

\begin{proof}
(i) Since $\lambda_j \ge 2/(j+2)$, $P_{0,k} \le \prod_{j=1}^k \frac{j}{j+2} = \frac{2}{(k+1)(k+2)}$. The weights telescope, $w_{i,k} = P_{i,k} - P_{i-1,k}$.
(ii) We first show $\lambda_{k-1}(1 - \lambda_k) \le \lambda_k$. This is clear if $\lambda_{k-1} = \lambda \le \lambda_k$. Otherwise $\lambda_{k-1} = 2/(k+1)$, and the inequality is equivalent to $\lambda_k(k+3) \ge 2$, which holds since $\lambda_k \ge 2/(k+2)$. As $w_{k,k} = \lambda_k$ and $w_{i,k} = (1-\lambda_k)w_{i,k-1}$ for $i < k$, induction on $k$ gives $\max_i w_{i,k} \le \lambda_k$. Each factor of $q_{i,k}$ is at most $1 - \lambda$, so $q_{i,k} \le (1-\lambda)^{k-i+1}$, and both sums are at most $\sum_{j \ge 1}(1-\lambda)^j \le \lambda^{-1}$.
(iii) By (i) and (ii), $\sum_i w_{i,k}^r \le (\max_i w_{i,k})^{r-1}\sum_i w_{i,k} \le \lambda_k^{r-1}$. Lemma~\ref{lem:vbe}(ii), with deterministic weights, and Jensen's inequality give the claim.
(iv) We have $\lambda_k^s \le \lambda^s + 2^s(k+2)^{-s}$ and $\sum_{k=1}^{K-1}(k+2)^{-s} \le K^{1-s}/(1-s)$, and $K^{-s} \le \lambda^s$. The second bound follows from $\sum_{k \ge 0}\frac{2}{(k+1)(k+2)} = 2$.
\end{proof}

\section{Composite Curvature Bound} \label{sec:curvature_appendix}

Integrating the Jacobian along the segment and using Assumption~\ref{assump:inner_smooth} gives, for $\vx, \vx + \vh \in \cX$,
\begin{equation} \label{eq:taylor_bound}
\|\vf(\vx + \vh) - \vf(\vx) - \nabla\vf(\vx)\vh\| = \Big\|\int_0^1 \big(\nabla\vf(\vx + t\vh) - \nabla\vf(\vx)\big)\vh \, dt\Big\| \le \tfrac{L}{2}\|\vh\|^2 .
\end{equation}

\begin{proposition}[Bounded composite curvature] \label{prop:curvature}
Under Assumptions~\ref{assump:domain}--\ref{assump:inner_smooth}, the constant $\mathcal{S}$ of \eqref{eq:curvature_def} satisfies $0 \le \mathcal{S} \le L_F L D^2$.
\end{proposition}

\begin{proof}
Taking $\vy = \vx$ in \eqref{eq:curvature_def} makes the bracket vanish, so $\mathcal{S} \ge 0$. For $\vx, \vy \in \cX$ and $\gamma \in (0,1]$, the point $\vy_\gamma = \vx + \vh$ with $\vh = \gamma(\vy - \vx)$ satisfies, by Lipschitzness of $F$ in $\vu$ and \eqref{eq:taylor_bound}, $F(\vf(\vy_\gamma), \vy_\gamma) - F(\vf(\vx) + \nabla\vf(\vx)\vh, \vy_\gamma) \le \tfrac{L_F L}{2}\gamma^2 D^2$. Multiplying by $2/\gamma^2$ gives the claim.
\end{proof}

\section{Full Statements: Progress and Tracking-Error Lemmas} \label{sec:lemmas_appendix}

\paragraph{Notation.} Let $\mathcal{F}_k$ be the $\sigma$-algebra generated by $\xi_0, \ldots, \xi_k$, so that $\vy_k$ is $\mathcal{F}_{k-1}$-measurable. Write $\vd_k := \vy_k - \vy_{k-1}$, so $\|\vd_k\| \le \gamma_{k-1} D$, and define the tracking errors and oracle errors
\begin{align*}
\bm{\delta}_{g,k} &:= \mV_k - \nabla \vf(\vy_k), & \bm{\delta}_{f,k} &:= \vz_k - \vf(\vy_k), \\
\bm\xi_{g,k} &:= \nabla\tilde\vf(\vy_k; \xi_k) - \nabla\vf(\vy_k), & \bm\xi_{f,k} &:= \tilde\vf(\vy_k; \xi_k) - \vf(\vy_k).
\end{align*}
By Assumption~\ref{assump:noise}, both oracle errors are conditionally mean zero given $\mathcal{F}_{k-1}$, with $\E[\|\bm\xi_{f,k}\|^{r_f} \mid \mathcal{F}_{k-1}] \le \sigma_f^{r_f}$ and $\E[\|\bm\xi_{g,k}\|^{r_g} \mid \mathcal{F}_{k-1}] \le \sigma_g^{r_g}$. We use the expected norms and their averages
\[
g_k := \E\|\bm{\delta}_{g,k}\|, \quad u_k := \E\|\bm{\delta}_{f,k}\|, \quad \bar g := \frac1K\sum_{k=0}^{K-1} g_k, \quad \bar u := \frac1K\sum_{k=0}^{K-1} u_k,
\]
where $g_0 \le \sigma_g$ and $u_0 \le \sigma_f$ for the one-sample initialization, and the noise constants $N_g := d_{r_g} C_{s_g}\sigma_g$ and $N_f := d_{r_f} C_{s_f}\sigma_f$ from Lemma~\ref{lem:kernel}. In Lemmas~\ref{lem:grad_tracker_main}--\ref{lem:func_taylor_main}, $\gamma_k \equiv \gamma$ and the tracker weights are $\beta_k = \max\{\beta, 2/(k+2)\}$ and $\rho_k = \max\{\rho, 2/(k+2)\}$ for $k \ge 1$, with floors $\beta, \rho \in [K^{-1}, 1]$.

\paragraph{Why first moments.} The Taylor recursion contains $\bm\delta_{g,k}\vd_k$, which need not have an $r_f$-th moment when $r_g < r_f$. We therefore bound expected norms throughout and apply moment inequalities only to the martingale sums of oracle noise, each at its own order.

\begin{lemma}[GLMO progress with tracking errors] \label{lem:glmo_main}
Under Assumptions~\ref{assump:domain}--\ref{assump:inner_smooth}, suppose $\vx_{k+1}$ minimizes the tracked model $\vx \mapsto F(\vz_k + \mV_k(\vx - \vy_k), \vx)$ over $\cX$ up to an additive error $\eta_k \ge 0$. Then, with $E_k := \|\bm{\delta}_{f,k}\| + D \|\bm{\delta}_{g,k}\|$,
\[
\varphi(\vy_{k+1}) \le \varphi(\vy_k) - \gamma_k \hat\Delta_k + \tfrac{\gamma_k^2}{2} \mathcal{S} + 2 \gamma_k L_F E_k + \gamma_k \eta_k.
\]
\end{lemma}

\begin{lemma}[Jacobian tracker] \label{lem:grad_tracker_main}
Under Assumptions~\ref{assump:inner_smooth} and \ref{assump:noise},
$\displaystyle \bar g \le \frac{2g_0}{K} + L_J D\frac{\gamma}{\beta} + N_g \beta^{s_g}$.
\end{lemma}

\begin{lemma}[Polyak function tracker] \label{lem:func_polyak_main}
Under Assumptions~\ref{assump:inner_smooth}, \ref{assump:noise} and \ref{assump:inner_lipschitz},
$\displaystyle \bar u \le \frac{2u_0}{K} + GD\frac{\gamma}{\rho} + N_f \rho^{s_f}$.
\end{lemma}

\begin{lemma}[Taylor function tracker] \label{lem:func_taylor_main}
Under Assumptions~\ref{assump:inner_smooth} and \ref{assump:noise}, for any Jacobian estimate $\mV_k$ used in the Variant~II update and any ordering of $r_f$ and $r_g$,
\begin{equation} \label{eq:taylor_convolution}
\bar u \le \frac{2u_0}{K} + D\frac{\gamma}{\rho}\,\bar g + \frac{L D^2 \gamma^2}{2\rho} + N_f \rho^{s_f}.
\end{equation}
\end{lemma}

The Taylor correction replaces the drift $GD\gamma/\rho$ of the Polyak tracker by the coupling $D\gamma\bar g/\rho$ and the remainder $LD^2\gamma^2/(2\rho)$. Both are smaller than $\gamma/\rho$ by a factor that vanishes with $K$, which allows a smaller floor $\rho$ and hence better averaging of heavy-tailed values.

\section{Proof of Lemma~\ref{lem:glmo_main}} \label{sec:proof_glmo}

\begin{proof}
Define the exact and tracked surrogates $l_k(\vx) := \vf(\vy_k) + \nabla \vf(\vy_k)(\vx - \vy_k)$ and $\tilde l_k(\vx) := \vz_k + \mV_k(\vx - \vy_k)$, and the models $m_k(\vx) := F(l_k(\vx), \vx)$ and $\tilde m_k(\vx) := F(\tilde l_k(\vx), \vx)$. Since $\|\cdot\|_{\mathrm{op}} \le \|\cdot\|_F$,
\begin{equation} \label{eq:uniform_surrogate_bound}
\|\tilde l_k(\vx) - l_k(\vx)\| \le \|\bm{\delta}_{f,k}\| + D \|\bm{\delta}_{g,k}\| = E_k, \qquad \text{hence} \qquad |\tilde m_k(\vx) - m_k(\vx)| \le L_F E_k \quad (\vx \in \cX).
\end{equation}

\textbf{Step 1.} Applying \eqref{eq:curvature_def} with $\vx \leftarrow \vy_k$, $\vy \leftarrow \vx_{k+1}$, $\gamma \leftarrow \gamma_k$, and noting $\vf(\vy_k) + \gamma_k \nabla \vf(\vy_k)(\vx_{k+1} - \vy_k) = l_k(\vy_{k+1})$,
\begin{equation} \label{eq:step1_curvature}
\varphi(\vy_{k+1}) \le m_k(\vy_{k+1}) + \tfrac{\gamma_k^2}{2} \mathcal{S}.
\end{equation}

\textbf{Step 2.} By Assumption~\ref{assump:outer}, $F$ is jointly convex and $l_k$ is affine, so $m_k$ is convex. Since $\vy_{k+1} = (1-\gamma_k)\vy_k + \gamma_k\vx_{k+1}$ and $m_k(\vy_k) = \varphi(\vy_k)$,
\begin{equation} \label{eq:step2_convexity}
m_k(\vy_{k+1}) \le (1 - \gamma_k) \varphi(\vy_k) + \gamma_k m_k(\vx_{k+1}).
\end{equation}

\textbf{Step 3.} Let $\bar \vx_{k+1} \in \argmin_{\vx \in \cX} m_k(\vx)$, which exists because $m_k$ is continuous on the compact set $\cX$; then $m_k(\bar\vx_{k+1}) = \varphi(\vy_k) - \hat\Delta_k$. By \eqref{eq:uniform_surrogate_bound} and the $\eta_k$-optimality of $\vx_{k+1}$ for $\tilde m_k$,
\[
m_k(\vx_{k+1}) \le \tilde m_k(\vx_{k+1}) + L_F E_k \le \tilde m_k(\bar \vx_{k+1}) + \eta_k + L_F E_k \le \varphi(\vy_k) - \hat\Delta_k + \eta_k + 2 L_F E_k.
\]

\textbf{Step 4.} Combining the three steps gives the claim.
\end{proof}

\section{Proof of Lemma~\ref{lem:grad_tracker_main}} \label{sec:proof_grad}

\begin{proof}
Let $\bm\Delta_i := \nabla \vf(\vy_{i-1}) - \nabla \vf(\vy_i)$, so $\|\bm\Delta_i\| \le L_J \|\vd_i\| \le L_J \gamma D$. Subtracting $\nabla \vf(\vy_k)$ from the Polyak update gives $\bm{\delta}_{g,k} = (1-\beta_k)(\bm{\delta}_{g,k-1} + \bm{\Delta}_k) + \beta_k \bm{\xi}_{g,k}$ for $k \ge 1$. With the kernel of Lemma~\ref{lem:kernel} for $\lambda_k = \beta_k$, unrolling gives
\[
\bm{\delta}_{g,k} = P_{0,k} \bm{\delta}_{g,0} + \sum_{i=1}^k q_{i,k} \bm\Delta_i + \sum_{i=1}^k w_{i,k} \bm\xi_{g,i}.
\]
Taking expected norms and using Lemma~\ref{lem:kernel}(ii)--(iii) at order $r_g$, $g_k \le P_{0,k} g_0 + L_J D\gamma/\beta + d_{r_g}\sigma_g \beta_k^{s_g}$ for $k \ge 1$; for $k = 0$ the bound holds since $P_{0,0} = 1$. Averaging over $k < K$ with Lemma~\ref{lem:kernel}(iv) gives the claim.
\end{proof}

\section{Proof of Lemma~\ref{lem:func_polyak_main}} \label{sec:proof_polyak}

\begin{proof}
Subtracting $\vf(\vy_k)$ from the Polyak update gives $\bm{\delta}_{f,k} = (1-\rho_k)(\bm{\delta}_{f,k-1} + \vf(\vy_{k-1}) - \vf(\vy_k)) + \rho_k \bm{\xi}_{f,k}$ with $\|\vf(\vy_{k-1}) - \vf(\vy_k)\| \le G\gamma D$. The proof of Lemma~\ref{lem:grad_tracker_main} applies with the $\rho$-kernel and order $r_f$.
\end{proof}

\section{Proof of Lemma~\ref{lem:func_taylor_main}} \label{sec:proof_taylor}

\begin{proof}
Define the Taylor remainder $\vR_k := \vf(\vy_k) - \vf(\vy_{k-1}) - \nabla\vf(\vy_k)\vd_k$; by \eqref{eq:taylor_bound} at $\vy_k$, $\|\vR_k\| \le \tfrac{L}{2}\gamma^2 D^2$. Since $\vf(\vy_{k-1}) + \mV_k \vd_k - \vf(\vy_k) = \bm\delta_{g,k}\vd_k - \vR_k$, subtracting $\vf(\vy_k)$ from the Variant II update gives, for $k \ge 1$,
\begin{equation}\label{eq:taylor_recursion_appendix}
\bm{\delta}_{f,k} = (1-\rho_k) \bm{\delta}_{f,k-1} + (1-\rho_k) \bm{\delta}_{g,k}\vd_k - (1-\rho_k) \vR_k + \rho_k \bm{\xi}_{f,k}.
\end{equation}
Unrolling with the $\rho$-kernel,
\[
\bm\delta_{f,k} = P_{0,k} \bm{\delta}_{f,0} + \sum_{i=1}^k q_{i,k} \big(\bm{\delta}_{g,i}\vd_i - \vR_i\big) + \sum_{i=1}^k w_{i,k} \bm\xi_{f,i}.
\]
Taking expected norms, $\E\|\bm\delta_{g,i}\vd_i\| \le \gamma D\, g_i$, Lemma~\ref{lem:kernel}(ii), and Lemma~\ref{lem:kernel}(iii) at order $r_f$ for the last sum give
\[
u_k \le P_{0,k} u_0 + D\gamma \sum_{i=1}^k q_{i,k}\, g_i + \frac{LD^2\gamma^2}{2\rho} + d_{r_f}\sigma_f\rho_k^{s_f}.
\]
Only the last sum uses a moment assumption, and no independence between $\bm\delta_{g,i}$ and $\bm\xi_{f,i}$ is needed. Averaging over $k < K$ and exchanging the order of summation,
$\frac1K\sum_{k=1}^{K-1}\sum_{i=1}^k q_{i,k} g_i = \frac1K\sum_{i=1}^{K-1} g_i \sum_{k=i}^{K-1} q_{i,k} \le \bar g/\rho$
by Lemma~\ref{lem:kernel}(ii). Lemma~\ref{lem:kernel}(iv) bounds the initial and noise terms.
\end{proof}

\section{Proof of Theorem~\ref{thm:nonconvex}} \label{sec:proof_nonconvex}

\subsection{Master bound}

With constant $\gamma$, $\|\vd_k\| \le \gamma D$. Summing Lemma~\ref{lem:glmo_main} (exact GLMO) over $k = 0, \ldots, K-1$, using $\varphi(\vy_K) \ge \varphi^*$, and dividing by $\gamma K$,
\begin{equation} \label{eq:telescope_main}
\E[\hat\Delta(\vy_R)] = \frac{1}{K}\sum_{k=0}^{K-1} \E[\hat\Delta_k] \le \frac{\Phi_0}{\gamma K} + \frac{\gamma\mathcal{S}}{2} + 2L_F\big(\bar u + D\, \bar g\big),
\end{equation}
where the equality holds because $R$ is uniform and independent of the run. With an inexact GLMO, see Corollary~\ref{cor:inexact}.

\subsection{Variant II}

Let $\theta := \theta_{\mathrm{II}}$. Since $\theta \le s_f$, the floor $\rho = K^{-\theta/s_f}$ is at least $K^{-1}$, and so is $\beta = K^{-(1-\theta)/(1+s_g)}$. The choice $\beta = \gamma^{1/(1+s_g)}$ gives $\gamma/\beta = \beta^{s_g} = \gamma^{a_g}$, so Lemma~\ref{lem:grad_tracker_main} yields $\bar g \le 2g_0/K + B_g\gamma^{a_g}$ with $B_g := L_J D + N_g$. Inserting this and Lemma~\ref{lem:func_taylor_main} into \eqref{eq:telescope_main},
\begin{equation} \label{eq:ncvx_expanded}
\E[\hat\Delta(\vy_R)] \le \frac{\Phi_0}{K\gamma} + \frac{\mathcal{S}}{2}\gamma + 2L_F\Big[\frac{2u_0}{K} + D\Big(1 + \frac{\gamma}{\rho}\Big)\Big(\frac{2g_0}{K} + B_g\gamma^{a_g}\Big) + \frac{LD^2\gamma^2}{2\rho} + N_f\rho^{s_f}\Big].
\end{equation}
The three terms in \eqref{eq:theta_def} are equivalent to
\begin{equation} \label{eq:theta_ineqs}
\theta \le a_g(1-\theta), \qquad \theta + \theta/s_f \le (1+a_g)(1-\theta), \qquad \theta \le s_f .
\end{equation}
With $\gamma = K^{-(1-\theta)}$ and $\rho = K^{-\theta/s_f}$ they give $(K\gamma)^{-1} = K^{-\theta}$, $\gamma^{a_g} \le K^{-\theta}$, $\gamma^{1+a_g}/\rho \le K^{-\theta}$ and $\rho^{s_f} = K^{-\theta}$. Moreover $\theta \le a_g/(1+a_g) \le 1/2$ gives $\gamma \le K^{-\theta}$; $\gamma^2/\rho \le \gamma^{1+a_g}/\rho$ since $a_g \le 1$; and $\gamma/(K\rho) \le \gamma$ since $\rho \ge K^{-1}$. Bounding each term of \eqref{eq:ncvx_expanded} gives $\E[\hat\Delta(\vy_R)] \le M_{\mathrm{II}}K^{-\theta_{\mathrm{II}}}$ with
\begin{equation} \label{eq:MII_explicit}
M_{\mathrm{II}} := \Phi_0 + \frac{\mathcal{S}}{2} + 2L_F\Big\{2u_0 + 4Dg_0 + 2DB_g + \frac{LD^2}{2} + N_f\Big\}.
\end{equation}

\paragraph{The exponent.} Write $\gamma = K^{-u}$ and $\rho = K^{-v}$ with $v \le 1$. The four terms $(K\gamma)^{-1}$, $\gamma^{a_g}$, $\gamma^{1+a_g}/\rho$ and $\rho^{s_f}$ of \eqref{eq:ncvx_expanded} are $\mathcal{O}(K^{-\theta})$ if and only if $\theta \le 1-u$, $\theta \le a_g u$, $\theta \le (1+a_g)u - v$ and $\theta \le s_f v$. Eliminating $u$ and $v$ leaves exactly the three conditions \eqref{eq:theta_ineqs}, so $\theta_{\mathrm{II}}$ is the best exponent obtainable from this bound; it is not a lower bound over algorithms. In terms of the moment orders,
\[
\theta_{\mathrm{II}} = \min\Big\{\frac{r_g-1}{3r_g-2},\ \frac{(r_f-1)(3r_g-2)}{7r_fr_g - 4r_f - 5r_g + 3},\ \frac{r_f-1}{r_f}\Big\},
\]
with denominator $(2r_f-1)(2r_g-1) + (r_f-1)(3r_g-2) > 0$. The first term is active if and only if $s_f \ge a_g$, that is $r_f \ge 2 - 1/r_g$, and the last if and only if $s_f \le a_g/(2+a_g)$.

\subsection{Variant I}

Let $q := \min\{a_g, a_f\}$ and $\gamma = K^{-1/(1+q)}$; the floors $\beta = \gamma^{1/(1+s_g)}$ and $\rho = \gamma^{1/(1+s_f)}$ are at least $\gamma \ge K^{-1}$. As above, $\bar g \le 2g_0/K + B_g\gamma^{a_g}$, and Lemma~\ref{lem:func_polyak_main} gives $\bar u \le 2u_0/K + B_f\gamma^{a_f}$ with $B_f := GD + N_f$. Since $\gamma^{a_g}, \gamma^{a_f} \le \gamma^q = K^{-\theta_{\mathrm{I}}}$, $(K\gamma)^{-1} = K^{-\theta_{\mathrm{I}}}$ and $\gamma \le K^{-\theta_{\mathrm{I}}}$, \eqref{eq:telescope_main} gives $\E[\hat\Delta(\vy_R)] \le M_{\mathrm{I}}K^{-\theta_{\mathrm{I}}}$ with
\begin{equation} \label{eq:MI_explicit}
M_{\mathrm{I}} := \Phi_0 + \frac{\mathcal{S}}{2} + 2L_F\big\{2u_0 + B_f + D(2g_0 + B_g)\big\}.
\end{equation}
Here $u_0 \le \sigma_f$, $g_0 \le \sigma_g$, and $\mathcal{S} \le L_F L D^2$ by Proposition~\ref{prop:curvature}. This completes the proof. \hfill\qed

\subsection{Constant tracker parameters}

\begin{remark}[Constant parameters] \label{rem:constant_params}
With constant weights $\beta_k \equiv \beta$ and $\rho_k \equiv \rho$, the kernels are geometric and the proofs above give $\bar g \le \frac{g_0}{K\beta} + L_JD\frac{\gamma}{\beta} + d_{r_g}\sigma_g\beta^{s_g}$ and $\bar u \le \frac{u_0}{K\rho} + D\frac{\gamma}{\rho}\bar g + \frac{LD^2\gamma^2}{2\rho} + d_{r_f}\sigma_f\rho^{s_f}$ for Variant II, and $\bar u \le \frac{u_0}{K\rho} + GD\frac{\gamma}{\rho} + d_{r_f}\sigma_f\rho^{s_f}$ for Variant I. Two consequences follow, with $b_g := L_J D + d_{r_g}\sigma_g$.
\begin{enumerate}[label=(\alph*)]
\item \emph{Separate orders.} For Variant II, let $\theta_0 := \min\{(r_g-1)/(3r_g-2),\, a_f\}$, $\gamma = \rho = K^{-(1-\theta_0)}$ and $\beta = \gamma^{1/(1+s_g)}$. Then $\E[\hat\Delta(\vy_R)] \le M_0 K^{-\theta_0}$ with $M_0 := \Phi_0 + \frac{\mathcal{S}}{2} + 2L_F\{u_0 + 2D(g_0 + b_g) + \frac{LD^2}{2} + d_{r_f}\sigma_f\}$. This keeps the threshold $r_f \ge 2 - 1/r_g$ but is weaker than Theorem~\ref{thm:nonconvex} for heavier value noise.
\item \emph{Equal orders.} If $r_f = r_g = r$, the constant parameters $\gamma = K^{-(2r-1)/(3r-2)}$ and $\beta = \rho = K^{-r/(3r-2)}$ give $\E[\hat\Delta(\vy_R)] \le M K^{-(r-1)/(3r-2)}$, with $M = \Phi_0 + \frac{\mathcal{S}}{2} + 2L_F\{u_0 + Dg_0 + Db_g + \frac{LD^2}{2} + d_r\sigma_f + D(g_0 + b_g)\}$ for Variant II and $M = \Phi_0 + \frac{\mathcal{S}}{2} + 2L_F\{u_0 + GD + d_r\sigma_f + D(g_0 + b_g)\}$ for Variant I. Indeed, the initialization terms are then $g_0/(K\beta) = g_0K^{-2\theta}$ and $u_0/(K\rho) = u_0K^{-2\theta}$ with $\theta = (r-1)/(3r-2)$, and the remaining terms scale as in the proof above.
\end{enumerate}
Without the initially larger weights, the initialization term $u_0/(K\rho)$ need not be of order $K^{-\theta_{\mathrm{II}}}$ for the small floor $\rho$ of Theorem~\ref{thm:nonconvex}; the schedule in the theorem avoids this without extra samples.
\end{remark}

\subsection{Unknown horizon}

\begin{remark}[Unknown horizon] \label{rem:unknown_horizon}
The schedules of Theorem~\ref{thm:nonconvex} depend on the horizon $K$. When $K$ is not known in advance, run epochs $j = 0, 1, 2, \ldots$: epoch $j$ runs Algorithm~\ref{alg:hybrid_sfw} from $\vy_0$ with fresh samples and the schedule of Theorem~\ref{thm:nonconvex} for the horizon $K_j := 2^j$, and draws $R_j$ uniformly from $\{0, \ldots, K_j - 1\}$, independently of the run. After $T \ge 1$ iterations in total, let $J$ be the largest index such that epochs $0, \ldots, J$ are complete, that is $2^{J+1} - 1 \le T$, and output $\vy_{R_J}$ from epoch $J$. Maximality of $J$ gives $2^{J+2} - 1 > T$, hence $K_J = 2^J > (T+1)/4$, and Theorem~\ref{thm:nonconvex} applied to epoch $J$ gives
\[
\E[\hat\Delta(\vy_{R_J})] \le M K_J^{-\theta} \le 4^{\theta} M\,(T+1)^{-\theta},
\]
with $M$ and $\theta$ of the corresponding variant. Since $\theta \le 1/4$, the factor $4^{\theta}$ is at most $\sqrt{2}$. Here $T$ counts outer iterations, each with one sample and one GLMO call, and the same argument applies to Theorems~\ref{thm:storm_rate} and~\ref{thm:hessian_rate}, with the factor $4^{\theta_{\mathrm{acc}}} \le 4^{1/3}$; under Theorem~\ref{thm:hessian_rate} each iteration also uses one Hessian sample. The restarts discard the iterates of earlier epochs; they serve only to remove the dependence on $K$ from the guarantee.
\end{remark}

\section{Proof of Theorem~\ref{thm:convex}} \label{sec:proof_convex}

We work with $h_k := \E[\varphi(\vy_k)] - \varphi^*$. Note $\|\vd_k\| \le \gamma_{k-1} D$ for $k \ge 1$.

\subsection{Primal gap inequality}

\begin{lemma}[Gap dominates primal error] \label{lem:convex_primal}
Under Assumptions~\ref{assump:outer} and \ref{assump:convex}, $\hat\Delta(\vy) \ge \varphi(\vy) - \varphi^*$ for every $\vy \in \cX$.
\end{lemma}

\begin{proof}
Let $\vx^* \in \argmin_{\cX} \varphi$ and $m_{\vy}(\vx) := F(\vf(\vy) + \nabla\vf(\vy)(\vx - \vy), \vx)$. By Assumption~\ref{assump:convex}, $\min_{\cX} m_{\vy} \le m_{\vy}(\vx^*) \le \varphi(\vx^*) = \varphi^*$, and $\hat\Delta(\vy) = \varphi(\vy) - \min_{\cX} m_{\vy}$.
\end{proof}

Combining with Lemma~\ref{lem:glmo_main} (exact GLMO):
\begin{equation} \label{eq:convex_recur}
h_{k+1} \le (1 - \gamma_k) h_k + \tfrac{\gamma_k^2}{2}\mathcal{S} + 2\gamma_k L_F\big(u_k + D\, g_k\big).
\end{equation}

\subsection{Jacobian tracking under the decreasing schedule}

For $r \in (1,2]$ set $b(r) := r/(2r-1)$, $q(r) := b(r)(r-1)$, and $k(r) := \lceil (4q(r))^{1/(1-b(r))}\rceil + 2$. For the Jacobian write $b_g := b(r_g)$, $q_g := q(r_g) = r_g a_g$, $k_g := k(r_g)$ and $C_g := k_g + 1$. For $k \ge 1$ and $t := k + k_g$,
\begin{equation} \label{eq:gamma_compare}
\gamma_{k-1} = \frac{2}{k+1} \le \frac{C_g}{t}, \qquad \text{equivalently} \qquad (k-1)(k_g - 1) \ge 0.
\end{equation}
For vectors $\vu, \vw$ in a Euclidean space, $\beta \in (0,1]$, and $r \in (1,2]$, convexity of $\tau \mapsto \tau^r$ gives $\|\vu + \vw\|^r \le (1+\beta)^{r-1}\|\vu\|^r + (1 + \beta^{-1})^{r-1}\|\vw\|^r$; with $(1-\beta)^r(1+\beta)^{r-1} \le 1-\beta$ and $(1-\beta)^r(1+\beta^{-1})^{r-1} \le 2^{r-1}\beta^{1-r}$,
\begin{equation} \label{eq:young_beta}
\|(1-\beta)(\vu + \vw)\|^r \le (1-\beta)\|\vu\|^r + 2^{r-1}\beta^{1-r}\|\vw\|^r .
\end{equation}

\begin{lemma}[Recursion to rate] \label{lem:induction}
Let $b \in (0,1)$, $q \in (0,1]$, and $k_* \ge 2$ with $2q k_*^{b-1} \le 1/2$. Let $v_k \ge 0$ satisfy $v_k \le (1 - t^{-b})v_{k-1} + M t^{-b-q}$ for all $k \ge 1$, with $t = k + k_*$. If $Q \ge \max\{k_*^q v_0,\, 2M\}$, then $v_k \le Q (k+k_*)^{-q}$ for all $k \ge 0$.
\end{lemma}

\begin{proof}
The case $k = 0$ is the choice of $Q$. For $k \ge 1$ we have $t \ge 3$, and $(t-1)^{-q} \le t^{-q}(1 + 2q/t)$ because $q \le 1$ and $1/t \le 1/2$. Assuming the claim at $k-1$,
\[
v_k \le Q t^{-q}(1 - t^{-b})(1 + 2q/t) + M t^{-b-q} \le Q t^{-q} + t^{-b-q}\big(M - Q + 2qQ\, t^{b-1}\big),
\]
and $2qt^{b-1} \le 2qk_*^{b-1} \le 1/2$, so the bracket is at most $M - Q/2 \le 0$.
\end{proof}

\begin{lemma}[Decreasing-schedule Jacobian tracking] \label{lem:convex_tracking}
Under Assumptions~\ref{assump:inner_smooth}, \ref{assump:noise} and $\beta_k = (k+k_g)^{-b_g}$, for all $k \ge 0$, $g_k \le Q_g (k+k_g)^{-a_g}$ with $Q_g := \bar Q_g^{1/r_g}$ and
\begin{equation} \label{eq:Qg_explicit}
\bar Q_g := \max\{k_g^{q_g} \sigma_g^{r_g},\ 2 M_g\}, \qquad M_g := 2^{r_g-1}(L_J D C_g)^{r_g} + c_{r_g} \sigma_g^{r_g}.
\end{equation}
\end{lemma}

\begin{proof}
Let $v_k := \E\|\bm\delta_{g,k}\|^{r_g}$, with $v_0 \le \sigma_g^{r_g}$, and $\bm\Delta_k := \nabla\vf(\vy_{k-1}) - \nabla\vf(\vy_k)$, so $\|\bm\Delta_k\| \le L_J D \gamma_{k-1}$. For $k \ge 1$, $\bm\delta_{g,k} = (1-\beta_k)(\bm\delta_{g,k-1} + \bm\Delta_k) + \beta_k\bm\xi_{g,k}$, where the first part is $\mathcal{F}_{k-1}$-measurable. Lemma~\ref{lem:vbe}(iii) and \eqref{eq:young_beta} give $v_k \le (1-\beta_k)v_{k-1} + 2^{r_g-1}\beta_k^{1-r_g}(L_J D\gamma_{k-1})^{r_g} + c_{r_g}\sigma_g^{r_g}\beta_k^{r_g}$. Since $b_g(r_g - 1) - r_g = -b_g r_g$ and $b_g r_g = b_g + q_g$, \eqref{eq:gamma_compare} gives $v_k \le (1 - t^{-b_g})v_{k-1} + M_g t^{-b_g - q_g}$. The choice of $k_g$ gives $2q_gk_g^{b_g-1} \le 1/2$, and Lemma~\ref{lem:induction} and Jensen's inequality, with $q_g/r_g = a_g$, give the claim.
\end{proof}

\subsection{Function tracking}

\begin{lemma}[Taylor tracker with $\rho_k = 2/(k+2)$] \label{lem:taylor_decreasing}
Let $p := \min\{a_g, s_f\}$. Under Assumptions~\ref{assump:inner_smooth}, \ref{assump:noise} and the schedule \eqref{eq:convex_schedule}, $u_k \le Q_f (k+2)^{-p}$ for all $k \ge 0$, with
\begin{equation} \label{eq:Qf_explicit}
Q_f := 2^p u_0 + \frac{2DQ_g}{2 - a_g} + 2LD^2 + 2^{s_f}d_{r_f}\sigma_f .
\end{equation}
\end{lemma}

\begin{proof}
For $\rho_i = \gamma_i = 2/(i+2)$ the kernel of \eqref{eq:taylor_recursion_appendix} is explicit. With $\mathsf{D}_k := (k+1)(k+2)$, $P_{0,k} = 2/\mathsf{D}_k$, $q_{i,k} = i(i+1)/\mathsf{D}_k$ and $w_{i,k} = 2(i+1)/\mathsf{D}_k$, so that $\sum_i w_{i,k} \le 1$ and $\max_i w_{i,k} = 2/(k+2)$. Since $\|\vd_i\| \le 2D/(i+1)$ and $\|\vR_i\| \le 2LD^2/(i+1)^2$, unrolling as in the proof of Lemma~\ref{lem:func_taylor_main} and applying Lemma~\ref{lem:vbe}(ii) at order $r_f$ to the noise sum give, for $k \ge 1$,
\[
u_k \le \frac{2u_0}{\mathsf{D}_k} + \frac{2D}{\mathsf{D}_k}\sum_{i=1}^k i\, g_i + \frac{2LD^2 k}{\mathsf{D}_k} + d_{r_f}\sigma_f\Big(\frac{2}{k+2}\Big)^{s_f}.
\]
By Lemma~\ref{lem:convex_tracking}, $g_i \le Q_g i^{-a_g}$, so $\sum_{i=1}^k i\,g_i \le Q_g (k+1)^{2-a_g}/(2-a_g)$ and the second term is at most $\frac{2DQ_g}{2-a_g}(k+2)^{-a_g}$. The first term is at most $2^p u_0 (k+2)^{-p}$, the third at most $2LD^2(k+2)^{-1}$, and the fourth equals $2^{s_f}d_{r_f}\sigma_f(k+2)^{-s_f}$. Since $p \le \min\{a_g, s_f, 1\}$, all four are at most their constants times $(k+2)^{-p}$. At $k = 0$ the term $2^pu_0 2^{-p} = u_0$ covers $u_0$.
\end{proof}

\begin{lemma}[Polyak tracker, Variant I] \label{lem:polyak_decreasing}
With $b_f := b(r_f)$, $q_f := q(r_f)$, $k_f := k(r_f)$, $C_f := k_f + 1$ and $\rho_k = (k+k_f)^{-b_f}$, under Assumptions~\ref{assump:inner_smooth}, \ref{assump:noise} and \ref{assump:inner_lipschitz}, $u_k \le Q_f^{\mathrm{I}}(k+k_f)^{-a_f}$ with $Q_f^{\mathrm{I}} := (\max\{k_f^{q_f}\sigma_f^{r_f},\, 2M_f^{\mathrm{I}}\})^{1/r_f}$ and $M_f^{\mathrm{I}} := 2^{r_f-1}(GDC_f)^{r_f} + c_{r_f}\sigma_f^{r_f}$.
\end{lemma}

\begin{proof}
The proof of Lemma~\ref{lem:convex_tracking} applies with $\vf(\vy_{k-1}) - \vf(\vy_k)$ in place of $\bm\Delta_k$, the order $r_f$, and $\gamma_{k-1} \le C_f/(k + k_f)$.
\end{proof}

\subsection{Primal rate}

Let $p := \min\{a_g, s_f\}$ and $Q_f$ as in \eqref{eq:Qf_explicit} for Variant II, or $p := \min\{a_g, a_f\}$ and $Q_f := Q_f^{\mathrm{I}}$ for Variant I. Since $k_g, k_f \ge 2$, the lemmas above give $u_k + Dg_k \le (Q_f + DQ_g)(k+2)^{-p}$. Substituting into \eqref{eq:convex_recur} and using $\gamma_k = 2/(k+2)$ and $p < 1$,
\[
h_{k+1} \le \Big(1 - \frac{2}{k+2}\Big) h_k + \frac{B}{(k+2)^{1+p}}, \qquad B := 2\mathcal{S} + 4 L_F (Q_f + D Q_g).
\]
We show $h_k \le A_{\mathrm{cvx}}(k+2)^{-p}$ by induction, with
\begin{equation}\label{eq:Acvx_def}
A_{\mathrm{cvx}} := \max\Big\{2^p \Phi_0,\ \frac{B}{2-p}\Big\}.
\end{equation}
The case $k = 0$ holds since $h_0 = \Phi_0$. If it holds at $k$, then $h_{k+1} \le A_{\mathrm{cvx}}(k+2)^{-p} - (2A_{\mathrm{cvx}} - B)(k+2)^{-1-p}$, and $2A_{\mathrm{cvx}} - B \ge pA_{\mathrm{cvx}}$. Convexity of $\tau \mapsto \tau^{-p}$ gives $(k+3)^{-p} \ge (k+2)^{-p} - p(k+2)^{-1-p}$, so $h_{k+1} \le A_{\mathrm{cvx}}(k+3)^{-p}$. Hence $\E[\varphi(\vy_K)] - \varphi^* \le A_{\mathrm{cvx}}(K+2)^{-p}$.

\subsection{Gap bound for a late iterate}

Let $K \ge 2$, $m := \lfloor K/2 \rfloor$ and $W := \sum_{k=m}^{K-1}\gamma_k$. Summing Lemma~\ref{lem:glmo_main} over $k = m, \ldots, K-1$ and using $h_K \ge 0$,
\[
\sum_{k=m}^{K-1} \gamma_k \E[\hat\Delta_k] \le h_m + \frac{\mathcal{S}}{2}\sum_{k=m}^{K-1}\gamma_k^2 + 2L_F(Q_f + DQ_g)\sum_{k=m}^{K-1}\gamma_k (k+2)^{-p}.
\]
Here $W \ge K(K+1)^{-1} \ge 2/3$, $\sum_{k \ge m}\gamma_k^2 \le 4/(m+1) \le 8/K$, and $m + 2 \ge K/2$. For $R$ drawn from $\{m, \ldots, K-1\}$ with probabilities $\gamma_k/W$, dividing by $W$ gives $\E[\hat\Delta(\vy_R)] \le A_{\mathrm{gap}} K^{-p}$ with
\begin{equation} \label{eq:Agap_def}
A_{\mathrm{gap}} := \tfrac{3}{2}\, 2^p A_{\mathrm{cvx}} + 6\,\mathcal{S} + 2^{1+p} L_F (Q_f + D Q_g).
\end{equation}
The same bound holds for $\min_{m \le k < K}\E[\hat\Delta_k]$. \hfill\qed

\begin{remark}[Last-iterate gap]
The primal bound also controls the gap of the last iterate, at a slower rate. Applying Lemma~\ref{lem:glmo_main} with exact trackers and step $\tau \in [0,1]$ toward a minimizer of the exact model gives $\varphi(\vy) - \varphi^* \ge \tau\hat\Delta(\vy) - \mathcal{S}\tau^2/2$, and optimizing over $\tau$ yields $\hat\Delta(\vy) \le 2(\varphi(\vy) - \varphi^*) + \sqrt{2\mathcal{S}(\varphi(\vy) - \varphi^*)}$. With the primal bound and Jensen's inequality, $\E[\hat\Delta(\vy_K)] = \mathcal{O}(K^{-p/2})$, and $\mathcal{O}(K^{-p})$ when $\mathcal{S} = 0$.
\end{remark}

\section{Inexact GLMO} \label{sec:inexact}

We consider two accuracy notions for the GLMO step. Write $\tilde m_k(\vx) := F(\vz_k + \mV_k(\vx - \vy_k), \vx)$ for the tracked model and $\tilde\Delta_k := \tilde m_k(\vy_k) - \min_{\cX}\tilde m_k$ for its gap. The point $\vx_{k+1}$ is
\begin{itemize}
\item \emph{$\eta_k$-accurate} if $\tilde m_k(\vx_{k+1}) \le \min_{\cX}\tilde m_k + \eta_k$, for some $\eta_k \ge 0$ that may be random;
\item \emph{$\delta$-relatively accurate} if $\tilde m_k(\vx_{k+1}) \le \tilde m_k(\vy_k) - (1-\delta)\tilde\Delta_k$, for some $\delta \in [0,1)$.
\end{itemize}
Both can be verified by any solver that returns a certified lower bound on $\min_{\cX}\tilde m_k$, for example column generation with a dual bound. Lemma~\ref{lem:glmo_main} covers the first notion. For the second we need the following variant.

\begin{lemma}[Progress with a relatively accurate GLMO] \label{lem:glmo_relative}
Under Assumptions~\ref{assump:domain}--\ref{assump:inner_smooth}, if $\vx_{k+1}$ is $\delta$-relatively accurate, then with $E_k$ as in Lemma~\ref{lem:glmo_main},
\[
\varphi(\vy_{k+1}) \le \varphi(\vy_k) - (1-\delta)\gamma_k \hat\Delta_k + \tfrac{\gamma_k^2}{2} \mathcal{S} + 4 \gamma_k L_F E_k .
\]
\end{lemma}

\begin{proof}
Steps 1 and 2 of the proof of Lemma~\ref{lem:glmo_main} give $\varphi(\vy_{k+1}) \le (1-\gamma_k)\varphi(\vy_k) + \gamma_k m_k(\vx_{k+1}) + \gamma_k^2\mathcal{S}/2$, with $m_k$ the exact model. By \eqref{eq:uniform_surrogate_bound}, $|\tilde m_k - m_k| \le L_F E_k$ on $\cX$, so $\tilde m_k(\vy_k) \le \varphi(\vy_k) + L_F E_k$ and $\tilde\Delta_k \ge \hat\Delta_k - 2L_F E_k$. Hence
\[
m_k(\vx_{k+1}) \le \tilde m_k(\vx_{k+1}) + L_F E_k \le \tilde m_k(\vy_k) - (1-\delta)\tilde\Delta_k + L_F E_k \le \varphi(\vy_k) - (1-\delta)\hat\Delta_k + 4L_F E_k,
\]
using $2(1-\delta) \le 2$. Substituting gives the claim.
\end{proof}

\begin{corollary}[Inexact GLMO] \label{cor:inexact}
Under the assumptions and step sizes of Theorems~\ref{thm:nonconvex} and~\ref{thm:convex}:
\begin{enumerate}[label=(\roman*)]
\item \emph{Additive accuracy, Theorem~\ref{thm:nonconvex}.} $\E[\hat\Delta(\vy_R)] \le M K^{-\theta} + \frac{1}{K}\sum_{k<K}\E[\eta_k]$, with $M$ and $\theta$ of the corresponding variant. In particular the rate is unchanged if $\E[\eta_k] \le c K^{-\theta}$, and $M$ increases by $c$.
\item \emph{Additive accuracy, Theorem~\ref{thm:convex}.} If $\E[\eta_k] \le C_\eta (k+2)^{-p}$ for all $k$, the two bounds of Theorem~\ref{thm:convex} hold with $A_{\mathrm{cvx}}$ computed from \eqref{eq:Acvx_def} with $B + 2C_\eta$ in place of $B$, and with $A_{\mathrm{gap}}$ given by \eqref{eq:Agap_def} for this $A_{\mathrm{cvx}}$, plus $2^p C_\eta$.
\item \emph{Relative accuracy, Theorem~\ref{thm:nonconvex}.} $\E[\hat\Delta(\vy_R)] \le M' K^{-\theta}/(1-\delta)$, where $M'$ is $M$ with $2L_F$ replaced by $4L_F$.
\item \emph{Relative accuracy, Theorem~\ref{thm:convex}.} If $\delta < 1 - p/2$, then $\E[\varphi(\vy_K)] - \varphi^* \le A_\delta (K+2)^{-p}$ with
`\[
A_\delta := \max\Big\{2^p\Phi_0,\ \frac{B_\delta}{2(1-\delta) - p}\Big\}, \qquad B_\delta := 2\mathcal{S} + 8L_F(Q_f + DQ_g),
\]
and the late-iterate index $R$ of Theorem~\ref{thm:convex} satisfies $\E[\hat\Delta(\vy_R)] \le \big(\tfrac32 2^p A_\delta + 6\mathcal{S} + 2^{2+p}L_F(Q_f + DQ_g)\big) K^{-p}/(1-\delta)$.
\end{enumerate}
At $r_f = r_g = 2$, (iv) allows any $\delta < 5/6$.
\end{corollary}

\begin{proof}
(i) Summing Lemma~\ref{lem:glmo_main} with the term $\gamma\eta_k$ adds $\gamma\sum_{k<K}\E[\eta_k]$ to the telescoped bound; dividing by $\gamma K$ gives the extra average in \eqref{eq:telescope_main}. (ii) The recursion \eqref{eq:convex_recur} gains $\gamma_k\E[\eta_k] \le 2C_\eta (k+2)^{-1-p}$, so $B$ becomes $B + 2C_\eta$ and the induction of Section~\ref{sec:proof_convex} is unchanged. In the late-iterate argument the extra term $\frac{1}{W}\sum_{k \ge m}\gamma_k \E[\eta_k] \le C_\eta (m+2)^{-p} \le 2^p C_\eta K^{-p}$. (iii) Summing Lemma~\ref{lem:glmo_relative} gives $(1-\delta)\frac{1}{K}\sum_{k<K}\E[\hat\Delta_k] \le \frac{\Phi_0}{\gamma K} + \frac{\gamma\mathcal{S}}{2} + 4L_F(\bar u + D\bar g)$, and the proof of Theorem~\ref{thm:nonconvex} bounds the right-hand side. (iv) With Lemma~\ref{lem:convex_primal}, Lemma~\ref{lem:glmo_relative} gives $h_{k+1} \le (1 - (1-\delta)\gamma_k)h_k + \frac{\gamma_k^2}{2}\mathcal{S} + 4\gamma_k L_F(u_k + Dg_k) \le \big(1 - \frac{2(1-\delta)}{k+2}\big)h_k + \frac{B_\delta}{(k+2)^{1+p}}$. If $h_k \le A_\delta(k+2)^{-p}$, then $h_{k+1} \le A_\delta(k+2)^{-p} - (2(1-\delta)A_\delta - B_\delta)(k+2)^{-1-p}$, and $2(1-\delta)A_\delta - B_\delta \ge pA_\delta$ by the choice of $A_\delta$; the tangent inequality of Section~\ref{sec:proof_convex} closes the induction. The late-iterate bound follows from the argument of Section~\ref{sec:proof_convex} with $4L_F$ in place of $2L_F$ and the factor $1-\delta$ on the left-hand side.
\end{proof}

\paragraph{Oracle cost.} The results above count GLMO calls; we do not analyse a particular inner solver. As an example with an ordinary LMO, suppose that $F = F(\vu)$ does not depend on $\vx$ and that its Moreau envelope $F_\mu(\vu) := \min_{\vv}\{F(\vv) + \|\vu - \vv\|^2/(2\mu)\}$ and its gradient can be computed. Then $0 \le F - F_\mu \le \mu L_F^2/2$ and $\nabla F_\mu$ is $1/\mu$-Lipschitz, so the smoothed tracked model $\vx \mapsto F_\mu(\vz_k + \mV_k(\vx - \vy_k))$ has a gradient that is $\|\mV_k\|_{\mathrm{op}}^2/\mu$-Lipschitz, and after $t$ Frank--Wolfe steps its suboptimality is at most $2D^2\|\mV_k\|_{\mathrm{op}}^2/(\mu(t+2))$. With $\mu = \eta/L_F^2$, the tracked model is solved to accuracy $\eta$ once $t + 2 \ge 4D^2L_F^2\|\mV_k\|_{\mathrm{op}}^2/\eta^2$. At finite variance, $\mV_k$ is a convex combination of Jacobian samples, so $\E\|\mV_k\|_F^2 \le J_{\max}^2 + \sigma_g^2$ with $J_{\max} := \sup_{\cX}\|\nabla\vf\|_F$. For the non-convex rate at $r_f = r_g = 2$, $K \asymp \epsilon^{-4}$ iterations with $\eta \asymp \epsilon$ then require $O(\epsilon^{-6})$ LMO calls in expectation, up to these problem constants and the cost of the smoothing oracle; this is the order of LiNASA+ICG \citep{xiao2022projection}, which requires a smooth $F$, while the number of oracle samples remains $O(\epsilon^{-4})$. This calculation needs the smoothing oracle and $F$ independent of $\vx$; for $r_g < 2$ the second-moment bound on $\mV_k$ is not available in general. In the experiments, column generation with a warm-started atom pool solves the GLMO directly.

\section{Proof of Corollary~\ref{cor:deterministic}} \label{sec:proof_deterministic}

\begin{proof}
With $\sigma_f = \sigma_g = 0$ the initialization is exact, and for $k \ge 1$ with $\beta_k = \rho_k = 1$ the updates give $\mV_k = \nabla\vf(\vy_k)$ and $\vz_k = \vf(\vy_k)$ (the Taylor term has prefactor $1 - \rho_k = 0$). Hence $E_k = 0$ for all $k \ge 0$, and Lemma~\ref{lem:glmo_main} gives $\varphi(\vy_{k+1}) \le \varphi(\vy_k) - \gamma_k\hat\Delta_k + \gamma_k^2\mathcal{S}/2$.

\emph{Convex case.} By Lemma~\ref{lem:convex_primal}, $h_{k+1} \le (1-\gamma_k)h_k + \gamma_k^2\mathcal{S}/2$. Since $\gamma_0 = 1$, $h_1 \le \mathcal{S}/2 \le \mathcal{S}$. If $h_k \le 2\mathcal{S}/(k+1)$ for some $k \ge 1$, then $h_{k+1} \le 2\mathcal{S}\big[\tfrac{k}{(k+1)(k+2)} + \tfrac{1}{(k+2)^2}\big] \le \tfrac{2\mathcal{S}}{k+2}$, since $\tfrac{k}{k+1} + \tfrac{1}{k+2} \le 1$.

\emph{Non-convex case.} Summing over $k = 0, \ldots, K-1$ gives $\sum_{k<K}\gamma_k\hat\Delta_k \le \Phi_0 + \tfrac{\mathcal{S}}{2}\sum_{k<K}\gamma_k^2$. With $\gamma_k = (k+1)^{-1/2}$, $\sum_{k<K}\gamma_k \ge \sqrt{K}$ and $\sum_{k<K}\gamma_k^2 \le 1 + \ln K$, so $\min_{k<K}\hat\Delta_k \le (\Phi_0 + \tfrac{\mathcal{S}}{2}(1 + \ln K))/\sqrt{K}$.
\end{proof}

\section{Acceleration via STORM} \label{sec:storm_extension}

The Taylor function tracker and its analysis are unchanged; only the Jacobian tracker is replaced by one whose drift is a martingale difference. We first give a generic bound. Throughout this section $\gamma_k \equiv \gamma$, $\beta_k = \max\{\beta, 2/(k+2)\}$ and $\rho_k = \max\{\rho, 2/(k+2)\}$ for $k \ge 1$, with floors in $[K^{-1}, 1]$.

\begin{lemma}[Corrected Jacobian tracking] \label{lem:corrected_tracker}
Let $(\mathcal{G}_k)$ be a filtration such that $\vy_k$ is $\mathcal{G}_{k-1}$-measurable and $(\bm\xi_{g,k})$ is a martingale difference sequence with respect to it. Suppose that for $k \ge 1$, $\bm\delta_{g,k} = (1-\beta_k)\bm\delta_{g,k-1} + (1-\beta_k)\bm\zeta_k + \beta_k\bm\xi_{g,k}$, where $\bm\zeta_k$ is $\mathcal{G}_k$-measurable, $\E[\bm\zeta_k \mid \mathcal{G}_{k-1}] = 0$ and $\E[\|\bm\zeta_k\|^{r_g} \mid \mathcal{G}_{k-1}] \le T^{r_g}D^{r_g}\gamma^{r_g}$ for some $T \ge 0$; $\bm\zeta_k$ and $\bm\xi_{g,k}$ may be correlated. Then
$\bar g \le 2g_0/K + d_{r_g} T D\gamma\beta^{-1/r_g} + N_g \beta^{s_g}$.
\end{lemma}

\begin{proof}
Unrolling with the $\beta$-kernel of Lemma~\ref{lem:kernel} gives $\bm\delta_{g,k} = P_{0,k}\bm\delta_{g,0} + \sum_{i \le k} q_{i,k}\bm\zeta_i + \sum_{i \le k} w_{i,k}\bm\xi_{g,i}$. Since $q_{i,k} \le (1-\beta)^{k-i+1}$, $\sum_i q_{i,k}^{r_g} \le \beta^{-1}$, and Lemma~\ref{lem:vbe}(ii) bounds the expected norm of the correction sum by $d_{r_g}TD\gamma\beta^{-1/r_g}$. Lemma~\ref{lem:kernel}(iii) bounds the noise sum. We combine the two sums by the triangle inequality, which requires no independence between them, and average with Lemma~\ref{lem:kernel}(iv).
\end{proof}

\begin{assumption}[$r_g$-average smoothness] \label{assump:avg_smooth}
There exists $\bar L > 0$ such that $\E_\xi\|\nabla\tilde\vf(\vx; \xi) - \nabla\tilde\vf(\vy; \xi)\|^{r_g} \le \bar L^{r_g} \|\vx - \vy\|^{r_g}$ for all $\vx, \vy \in \cX$, with the Frobenius norm on the left.
\end{assumption}

For $k \ge 1$, the STORM Jacobian update reuses $\xi_k$ at $\vy_k$ and $\vy_{k-1}$:
\begin{equation} \label{eq:storm_grad}
\mV_k = \nabla\tilde\vf(\vy_k; \xi_k) + (1-\beta_k)[\mV_{k-1} - \nabla\tilde\vf(\vy_{k-1}; \xi_k)].
\end{equation}

\begin{lemma}[STORM Jacobian tracking error] \label{lem:storm_tracker}
Under Assumptions~\ref{assump:inner_smooth}, \ref{assump:noise} and \ref{assump:avg_smooth}, the update \eqref{eq:storm_grad} satisfies the hypothesis of Lemma~\ref{lem:corrected_tracker} with $\mathcal{G}_k = \mathcal{F}_k$ and $T = 2\bar L$.
\end{lemma}

\begin{proof}
Subtracting $\nabla\vf(\vy_k)$ from \eqref{eq:storm_grad} gives the recursion of Lemma~\ref{lem:corrected_tracker} with $\bm\zeta_k := [\nabla\tilde\vf(\vy_k; \xi_k) - \nabla\tilde\vf(\vy_{k-1}; \xi_k)] - [\nabla\vf(\vy_k) - \nabla\vf(\vy_{k-1})]$, which is conditionally mean zero because both query points are $\mathcal{F}_{k-1}$-measurable. The centering inequality $\E\|\bm{Z} - \E\bm{Z}\|^r \le 2^{r-1}(\E\|\bm{Z}\|^r + \|\E\bm{Z}\|^r) \le 2^r\E\|\bm{Z}\|^r$, applied conditionally, and Assumption~\ref{assump:avg_smooth} give $\E[\|\bm\zeta_k\|^{r_g} \mid \mathcal{F}_{k-1}] \le 2^{r_g}\bar L^{r_g}D^{r_g}\gamma^{r_g}$.
\end{proof}

\begin{theorem}[STORM acceleration] \label{thm:storm_rate}
Under Assumptions~\ref{assump:domain}--\ref{assump:noise} and \ref{assump:avg_smooth}, use the STORM update \eqref{eq:storm_grad} with the Variant II function tracker. Let
\[
\theta_{\mathrm{acc}} := \min\Big\{\frac{s_g}{1+s_g},\ \frac{s_f(1+s_g)}{1+s_f(2+s_g)},\ s_f\Big\},
\]
and choose the floors $\gamma = \beta = K^{-(1-\theta_{\mathrm{acc}})}$ and $\rho = K^{-\theta_{\mathrm{acc}}/s_f}$. Then, with $R$ uniform on $\{0, \ldots, K-1\}$, $\E[\hat\Delta(\vy_R)] \le M_{\mathrm{acc}}\, K^{-\theta_{\mathrm{acc}}}$, where
\begin{equation} \label{eq:Macc_def}
M_{\mathrm{acc}} := \Phi_0 + \frac{\mathcal{S}}{2} + 2L_F\Big\{2u_0 + 4Dg_0 + 2DB_{\mathrm{acc}} + \frac{LD^2}{2} + N_f\Big\}, \qquad B_{\mathrm{acc}} := 2d_{r_g}\bar L D + N_g.
\end{equation}
For $r_f = r_g = r$, $\theta_{\mathrm{acc}} = (r-1)/(2r-1)$, which is $1/3$ at $r = 2$. The exponent $s_g/(1+s_g)$ is retained whenever $r_f \ge r_g$.
\end{theorem}

\begin{proof}
With $\beta = \gamma$, Lemmas~\ref{lem:corrected_tracker} and \ref{lem:storm_tracker} give $\bar g \le 2g_0/K + B_{\mathrm{acc}}\gamma^{s_g}$, since $\gamma\beta^{-1/r_g} = \beta^{s_g} = \gamma^{s_g}$. This is the bound on $\bar g$ used in the proof of Theorem~\ref{thm:nonconvex}(i), with $a_g$ replaced by $s_g$ and $B_g$ by $B_{\mathrm{acc}}$, and Lemma~\ref{lem:func_taylor_main} applies unchanged. That proof used only $\theta \le 1/2$ and $a_g \le 1$, which hold here because $\theta_{\mathrm{acc}} \le s_g/(1+s_g) \le 1/3$ and $s_g \le 1/2$, and the floors are at least $K^{-1}$ because $\theta_{\mathrm{acc}} \le s_f$. For $r_f = r_g$ the first two terms coincide and equal $s_g/(1+s_g) = (r-1)/(2r-1)$; the first term is active if and only if $s_f \ge s_g$.
\end{proof}

\section{Hessian-Corrected Momentum} \label{sec:hessian_extension}

We present an alternative acceleration through a randomized Hessian-vector correction and relate its assumption to average smoothness. For a linear map $\mH : \R^d \to \R^{n \times d}$ we use $\|\mH\|_{2 \to F} := \sup_{\|\vh\| \le 1}\|\mH[\vh]\|_F$.

\begin{assumption}[Stochastic Hessian oracle] \label{assump:hessian_var}
$\vf$ is twice continuously differentiable on a neighborhood of $\cX$, with $\|\nabla^2\vf(\vx)\|_{2 \to F} \le L_J$ for $\vx \in \cX$. A Hessian oracle, queried with a fresh sample independent of the past, returns $\nabla^2\tilde\vf(\vx; \zeta)$ with $\E_\zeta[\nabla^2\tilde\vf(\vx; \zeta)] = \nabla^2\vf(\vx)$ and $\E_\zeta\|\nabla^2\tilde\vf(\vx; \zeta) - \nabla^2\vf(\vx)\|_{2 \to F}^{r_g} \le \sigma_H^{r_g}$ for all $\vx \in \cX$.
\end{assumption}

\subsection{Bounded Hessian noise implies average smoothness}

\begin{proposition}[Hessian noise implies average smoothness] \label{prop:hessian_implies_avg}
Suppose Assumption~\ref{assump:hessian_var} holds and the Hessian oracle is the Hessian of the same sample functions: for almost every $\xi$, $\tilde\vf(\cdot; \xi)$ is twice continuously differentiable on a neighborhood of $\cX$ with $\nabla^2\tilde\vf(\cdot; \xi) = \nabla_{\vx}[\nabla\tilde\vf(\cdot;\xi)]$. Then Assumption~\ref{assump:avg_smooth} holds with $\bar L = H_* := \{2^{r_g-1}(\sigma_H^{r_g} + L_J^{r_g})\}^{1/r_g}$.
\end{proposition}

\begin{proof}
First, $\E_\xi\|\nabla^2\tilde\vf(\vz; \xi)\|_{2 \to F}^{r_g} \le 2^{r_g-1}(\sigma_H^{r_g} + L_J^{r_g}) = H_*^{r_g}$ for all $\vz \in \cX$. Second, for almost every $\xi$,
$\nabla\tilde\vf(\vx;\xi) - \nabla\tilde\vf(\vy;\xi) = \int_0^1 \nabla^2\tilde\vf(\vy + \tau(\vx - \vy); \xi)[\vx - \vy]\,d\tau$. Taking Frobenius norms, applying Jensen's inequality to the uniform measure on $[0,1]$, and taking expectations with Fubini's theorem,
\[
\E_\xi\|\nabla\tilde\vf(\vx; \xi) - \nabla\tilde\vf(\vy; \xi)\|^{r_g} \le \|\vx - \vy\|^{r_g} \int_0^1 \E_\xi\|\nabla^2\tilde\vf(\vy + \tau(\vx - \vy); \xi)\|_{2\to F}^{r_g}\,d\tau \le H_*^{r_g} \|\vx - \vy\|^{r_g}. \qedhere
\]
\end{proof}

The proposition is an implication only. The deterministic scalar function $f(x) = \tfrac12 x|x|$ has the $1$-Lipschitz derivative $|x|$, so its exact oracle satisfies Assumption~\ref{assump:avg_smooth}, but $f$ is not twice differentiable at $0$. The implication also concerns a single sample field: an unbiased Hessian estimator drawn independently of the gradient oracle places no constraint on the smoothness of that gradient oracle.

\subsection{Hessian-corrected algorithm and convergence}

For $k \ge 1$, let $\alpha_k \sim \mathrm{U}(0,1)$ and a Hessian sample $\upsilon_k$ be drawn independently of each other, of $\xi_k$, and of the past. Since $\vy_k$ now depends on $\alpha_i$ and $\upsilon_i$ for $i < k$, we use the filtration $\mathcal{H}_k := \sigma(\xi_0, \ldots, \xi_k, \alpha_1, \ldots, \alpha_k, \upsilon_1, \ldots, \upsilon_k)$, so that $\vy_k$ is $\mathcal{H}_{k-1}$-measurable. Define $\tilde\mH_k := \nabla^2\tilde\vf(\vy_{k-1} + \alpha_k\vd_k; \upsilon_k)$ and the update
\begin{equation} \label{eq:hessian_grad}
\mV_k = (1-\beta_k)\mV_{k-1} + \beta_k\nabla\tilde\vf(\vy_k; \xi_k) + (1-\beta_k)\tilde\mH_k[\vd_k].
\end{equation}

\begin{lemma}[Unbiasedness of the Hessian-vector correction] \label{lem:hessian_unbiased}
Under Assumption~\ref{assump:hessian_var}, $\E[\tilde\mH_k[\vd_k] \mid \mathcal{H}_{k-1}] = \nabla\vf(\vy_k) - \nabla\vf(\vy_{k-1})$.
\end{lemma}

\begin{proof}
Conditionally on $\mathcal{H}_{k-1}$ and $\alpha_k$, unbiasedness of the Hessian oracle gives $\E[\tilde\mH_k \mid \mathcal{H}_{k-1}, \alpha_k] = \nabla^2\vf(\vy_{k-1} + \alpha_k\vd_k)$. Averaging over $\alpha_k$ and using the fundamental theorem of calculus, $\int_0^1 \nabla^2\vf(\vy_{k-1} + \tau\vd_k)[\vd_k]\,d\tau = \nabla\vf(\vy_k) - \nabla\vf(\vy_{k-1})$.
\end{proof}

\begin{lemma}[Hessian-corrected Jacobian tracking error] \label{lem:hessian_track}
Under Assumptions~\ref{assump:inner_smooth}, \ref{assump:noise} and \ref{assump:hessian_var}, the update \eqref{eq:hessian_grad} satisfies the hypothesis of Lemma~\ref{lem:corrected_tracker} with $\mathcal{G}_k = \mathcal{H}_k$ and $T = 2H_*$.
\end{lemma}

\begin{proof}
Subtracting $\nabla\vf(\vy_k)$ from \eqref{eq:hessian_grad} gives the recursion of Lemma~\ref{lem:corrected_tracker} with $\bm\zeta_k := \tilde\mH_k[\vd_k] - (\nabla\vf(\vy_k) - \nabla\vf(\vy_{k-1}))$, a martingale difference by Lemma~\ref{lem:hessian_unbiased}. By the centering inequality of Lemma~\ref{lem:storm_tracker} and $\E\|\nabla^2\tilde\vf(\vz;\upsilon)\|_{2\to F}^{r_g} \le H_*^{r_g}$, $\E[\|\bm\zeta_k\|^{r_g} \mid \mathcal{H}_{k-1}] \le 2^{r_g} H_*^{r_g} D^{r_g}\gamma^{r_g}$.
\end{proof}

\begin{theorem}[Hessian-corrected acceleration] \label{thm:hessian_rate}
Under Assumptions~\ref{assump:domain}--\ref{assump:noise} and \ref{assump:hessian_var}, with the update \eqref{eq:hessian_grad}, the Variant II function tracker and the floors of Theorem~\ref{thm:storm_rate}, the bound of Theorem~\ref{thm:storm_rate} holds with $\bar L$ replaced by $H_*$ in $M_{\mathrm{acc}}$.
\end{theorem}

\begin{proof}
Identical to the proof of Theorem~\ref{thm:storm_rate}, with Lemma~\ref{lem:hessian_track} in place of Lemma~\ref{lem:storm_tracker}.
\end{proof}

\begin{remark}[STORM and Hessian correction]
Both corrections give the exponent $\theta_{\mathrm{acc}}$. STORM evaluates the sample Jacobian at two points and needs average smoothness of the sample field. The Hessian correction needs an unbiased Hessian oracle, one extra sample, and one Hessian contraction $\tilde\mH_k[\vd_k]$ per step, an $n \times d$ matrix of component Hessian-vector products whose cost depends on the oracle, so the single-sample count of the base method does not apply to it. When the Hessian oracle is the Hessian of the same sample functions, Proposition~\ref{prop:hessian_implies_avg} shows that its assumption implies average smoothness.
\end{remark}

\section{Additional Examples} \label{sec:examples_appendix}

\begin{example}[Fairness-penalized learning]
Beyond the examples in the main text, Algorithm~\ref{alg:hybrid_sfw} applies to a penalized form of fair ERM \citep{agarwal2018reductions, donini2018empirical}, which penalizes the disparity between group losses:
\[
\min_{\vx \in \cX} f_0(\vx) + \lambda \max_{1 \le i \le n}|f_i(\vx) - f_0(\vx)| + \Xi(\vx).
\]
The outer function $F(\vu, \vx) = u_0 + \lambda\max_i|u_i - u_0| + \Xi(\vx)$ depends explicitly on $\vx$. For $\lambda \ge 0$ and convex $\Xi$ it is jointly convex and $(1 + \sqrt{2}\lambda)$-Lipschitz in $\vu$, but not monotone in $\vu$. With auxiliary variables for the absolute values, the GLMO is a linear program when $\cX$ is polyhedral and $\Xi$ is polyhedral.
\end{example}

\begin{example}[Exact penalty for stochastic constraints]
For safe RL or fairness-constrained classification, one can minimize an expected loss subject to expectation constraints via the exact penalty formulation \citep{cotter2019two}:
\[
\min_{\vx \in \cX} f_0(\vx) + \langle \vc, \vx \rangle + \lambda \sum_{i=1}^n \max(0, f_i(\vx) - \tau_i).
\]
For $\lambda \ge 0$ the outer function is jointly convex and $\sqrt{1 + n\lambda^2}$-Lipschitz in $\vu$. When $f_0, \ldots, f_n$ are convex, strong duality holds, and a dual optimum $\bm\lambda^\star$ is attained, every $\lambda > \|\bm\lambda^\star\|_\infty$ makes the penalty exact: its minimizers are the solutions of the constrained problem. Without convexity, such statements are local, and for smaller $\lambda$ the penalized minimizer can be infeasible. With auxiliary variables for the hinges, the GLMO is a linear program when $\cX$ is polyhedral.
\end{example}

\end{document}